\documentclass[11pt]{article}
\usepackage{amssymb,amsmath,amsthm,enumerate}
\usepackage{tikz}
\usepackage{fullpage}
\newtheorem{theorem}{Theorem}[section]
\newtheorem{corollary}[theorem]{Corollary}
\newtheorem{lemma}[theorem]{Lemma}
\newtheorem{question}[theorem]{Question}
\newtheorem{problem}[theorem]{Problem}
\newtheorem{prop}[theorem]{Proposition}
\newtheorem{definition}[theorem]{Definition}
\newtheorem{conj}[theorem]{Conjecture}
\theoremstyle{definition}

\newtheorem{claim}[theorem]{Claim}

\theoremstyle{definition}

\theoremstyle{definition}

\theoremstyle{definition}

\theoremstyle{definition}

\theoremstyle{definition}

\theoremstyle{definition}

\newcommand\ex{\ensuremath{\mathrm{ex}}}

\newcommand{\ep}{\varepsilon}

\newcommand{\al}{\alpha}
\newcommand{\de}{\delta}

\newcommand{\ga}{\gamma}

\newcommand{\cY}{\mathcal{Y}}
\newcommand{\cU}{\mathcal{U}}
\newcommand{\cH}{\mathcal{H}}

\newcommand{\cP}{\mathcal{P}}

\newcommand{\Ho}{\H{o}}

\newcommand{\NIM}{\textrm{NIM}}
\newcommand{\nim}{\mathrm{nim}}

\newcommand{\nib}[1]{\noindent{\textbf{#1}}}

\newif\ifnotesw\noteswtrue
\newcommand{\comment}[1]{\ifnotesw $\blacktriangleright$\ {\sf #1}\ 
  $\blacktriangleleft$ \fi}
\newcommand{\hide}[1]{}
\usepackage{bbm}
\newcommand{\I}[1]{{\mathbbm #1}}
\newcommand{\dedit}{\delta_{\mathrm{edit}}}
\newcommand{\K}[1]{K[#1]} % complete partite graph

\title{Edges not in any monochromatic copy of a fixed graph}

\author{Hong Liu \and Oleg Pikhurko \and Maryam Sharifzadeh
\thanks{Mathematics Institute and DIMAP, University of Warwick, Coventry, CV4 7AL, UK.  Email addresses: {\tt \{h.liu.9, o.pikhurko, m.sharifzadeh\}@warwick.ac.uk}. H.L.\ was supported by the Leverhulme Trust Early Career Fellowship~ECF-2016-523. O.P.\ was supported by ERC grant~306493 and EPSRC grant EP/K012045/1. M.Sh.\ was supported by ERC grant~306493 and Marie Curie Individual Fellowship 752426.}
}

\begin{document}
	\maketitle
	
\begin{abstract}
For a sequence $(H_i)_{i=1}^k$ of graphs, let $\nim(n;H_1,\ldots, H_k)$ denote the maximum number of edges not contained in any monochromatic copy of $H_i$ in colour $i$, for any colour $i$, over all $k$-edge-colourings of~$K_n$. 

When each $H_i$ is connected and non-bipartite, we introduce a variant of Ramsey number that determines the limit of $\nim(n;H_1,\ldots, H_k)/{n\choose 2}$ as $n\to\infty$ and prove the corresponding stability result. Furthermore, if each $H_i$ is what we call \emph{homomorphism-critical} (in particular if each $H_i$ is a clique), then we determine $\nim(n;H_1,\ldots, H_k)$ exactly for all  sufficiently large~$n$. The special case $\nim(n;K_3,K_3,K_3)$ of our result answers a question of Ma.

For bipartite graphs, we mainly concentrate on the two-colour symmetric case (i.e., when $k=2$ and $H_1=H_2$). It is trivial to see that $\nim(n;H,H)$ is at least $\ex(n,H)$, the maximum size of an $H$-free graph on $n$ vertices.
Keevash and Sudakov showed that equality holds if $H$ is the $4$-cycle
and $n$ is large; recently Ma extended their result to an infinite family of bipartite graphs. We provide a larger family of bipartite graphs for which $\nim(n;H,H)=\ex(n,H)$. For a general bipartite graph $H$, we show that $\nim(n;H,H)$ is always within a constant additive error from $\ex(n,H)$, i.e.,~$\nim(n;H,H)= \ex(n,H)+O_H(1)$.
\end{abstract}

%%%%%%%%%%%%%%%%%%%%%%%%%%%%%%%%%%%%%%%%%%%%%%%%%%%%%
%%%%%%%%%%%%%%%%%%%%%%%%%%%%%%%%%%%%%%%%%%%%%%%%%%%%%
\section{Introduction}

Many problems of extremal graph theory ask for (best possible) conditions that guarantee the existence of a given `forbidden' graph. Two prominent examples of this kind are the Tur\'an function and Ramsey numbers. %Let us recall the special cases of these two extremal functions that we will need in this paper. 
Recall that, for a graph $H$ and an integer $n$, the \emph{Tur\'an function} $\ex(n,H)$ is the maximum size of an $n$-vertex $H$-free graph.
Let $K_t$ denote the complete graph on $t$ vertices. The famous theorem of
Tur\'an~\cite{Turan41} states that the unique maximum $K_{r+1}$-free graph of order $n$ is the \emph{Tur\'an graph} $T(n,r)$,
the 
complete balanced $r$-partite graph. Thus $\ex(n,K_{k+1})=t(n,r)$, where
we denote $t(n,r):=e(T(n,r))$.
For a sequence $a_1,\dots,a_k$ of integers, the \emph{Ramsey number} $R(a_1,\dots,a_k)$ is the minimum $R$ such that for every edge-colouring of $K_R$ with colours from $[k]:=\{1,\dots,k\}$, there is a colour-$i$ copy of $K_{a_i}$ for some $i\in [k]$. The fact that $R$ exists (i.e.,\ is finite) was first established by Ramsey~\cite{Ramsey30} and then independently rediscovered by Erd\H os and Szekeres~\cite{ErdosSzekeres35}.
%(For the purposes of this paper, we find it more convenient to use this definition of Ramsey numbers, e.g.\ by 1 less than the standard one.) 
Both of these problems motivated a tremendous amount of research, see e.g.\ the recent surveys by 
%Thomason~\cite{Thomason02}, 
Conlon, Fox and Sudakov~\cite{ConlonFoxSudakov15},
F\"uredi and Simonovits~\cite{FurediSimonovits13},  
Keevash~\cite{Keevash11}, 
Radziszowski~\cite{Radziszowski:ds} and
Sudakov~\cite{Sudakov10imc}.

A far-reaching generalisation is to ask for the number of guaranteed forbidden subgraphs. For the Tur\'an function this gives the famous \emph{Erd\H os-Rademacher problem} that goes back to Rademacher (1941; unpublished): what is the minimum number of copies of $H$ in a graph of given order $n$ and size $m>\ex(n,H)$? This problem was revived by Erd\Ho s~\cite{Erdos55,Erdos62} in the 1950--60s. Since then it continues to be a very active area of research, for some recent results see e.g.~\cite{ConlonFoxSudakov10,ConlonKimLeeLee15arxiv,ConlonLee17,Hatami10,KimLeeLee16,LiSzegedy11arxiv,Mubayi10am,Nikiforov11,PikhurkoYilma17,Razborov08,Reiher16,Szegedy14arxiv:v3}.
The analogous question for Ramsey numbers, known as the \emph{Ramsey multiplicity problem}, was introduced by Erd\H os~\cite{Erdos62a} in 1962 and is wide open, see e.g.~\cite{Conlon12,CKPSTY,FranekRodl92,Giraud79,JaggerStovicekThomason96,Sperfeld11arxiv,Thomason89,Thomason97c}.

A less studied but still quite natural question is to maximise the number of edges that do not belong to any forbidden subgraph. Such problems in the Tur\'an context (where we are given the order $n$ and the size $m>\ex(n,H)$ of a graph $G$) were studied in~\cite{ErdosFaudreeRousseau92dm,FurediMaleki14arxiv,GruslysLetzter16arxiv,GrzesikHuVolec16arxiv}.
In the Ramsey context, a problem of this type seems to  have been first posed by Erd\H os, Rousseau, and Schelp (see~\cite[Page~84]{Erdos97}). Namely, they considered  the maximum number of edges not contained in any monochromatic triangle in a $2$-edge-colouring of $K_n$. Also, Erd\H os~\cite[Page~84]{Erdos97} wrote that \emph{``many further related questions can be asked''}. Such questions  will be the focus of this paper.

Let us provide a rather general definition. Suppose that we have fixed a sequence of graphs $H_1,\ldots, H_k$. For a $k$-edge-colouring $\phi$ of $K_n$, let $\NIM(\phi)$ consist of all \NIM-\emph{edges}, that is, those edges of $K_n$ that are not contained in any colour-$i$ copy of $H_i$ for any $i\in [k]$. In other words, $\NIM(\phi)$ is the complement (with respect to $E(K_n)$) of the union over $i\in [k]$ of the edge-sets of $H_i$-subgraphs of colour-$i$.
% We also refer to the spanning subgraph  of~$K_n$ induced by NIM-edges as the \emph{NIM-subgraph}. 
Define 
 $$\nim(n;H_1,\ldots,H_k):=\max_{\phi:E(K_n)\to [k]} |\NIM(\phi)|,
 $$ 
 to be the maximum possible number of $\NIM$-edges in a $k$-edge-colouring of $K_n$. If all $H_i$'s are the same graph $H$, we will write $\nim_k(n;H)$ instead. Note that for $k=2$ by taking one colour-class to be a maximum $H$-free graph, we have $\nim_2(n;H)\ge \ex(n,H)$. In (\cite[Page~84]{Erdos97}), Erd\H os mentioned that together with Rousseau and Schelp, they  showed that in fact 
 \begin{equation}\label{eq:=}
 \nim_2(n;H)=\ex(n,H),\quad\mbox{for all $n\ge n_0(H)$,}
 \end{equation}
 when $H=K_3$ is the triangle. As observed by Alon (see~\cite[Page 42]{KeevashSudakov04jctb}), this also follows from an earlier paper of Pyber~\cite{Pyber86}. Keevash and Sudakov~\cite{KeevashSudakov04jctb} showed that~\eqref{eq:=} holds when $H$ is an arbitrary clique $K_t$ (or, more generally, when $H$ is \emph{edge-colour-critical}, that is, the removal of some edge $e\in E(H)$ decreases the chromatic number) as well as when $H$ is the 4-cycle $C_4$ (and $n\ge 7$). They~\cite[Problem~5.1]{KeevashSudakov04jctb} also posed the following problem.
 
 %proposed the following problem.
\begin{problem}[Keevash and Sudakov~\cite{KeevashSudakov04jctb}]\label{prob-KS}
 Does~\eqref{eq:=} hold for every graph $H$?
 %	Let $H$ be a fixed graph $H$. Is it true that for $n$ sufficiently large, $\nim_2(n;H)=\ex(n,H)$?
 \end{problem}
 
In a recent paper, Ma~\cite{Ma17jctb} answered Problem~\ref{prob-KS} in the affirmative for the infinite family of reducible bipartite graphs, where a bipartite graph $H$ is called {\it reducible} if it contains a vertex $v\in V(H)$ such that $H-v$ is connected and $\ex(n,H-v)=o(\ex(n,H))$ as $n\to\infty$.
% Let $T(n,r)$ be the complete $r$-partite $n$-vertex Tur\'an graph with part sizes as equal as possible and let $t(n,r):=e(T(n,r))$. 
Ma~\cite{Ma17jctb} also studied the case of $k\ge 3$ colours and raised the following question.
 \begin{question}[Ma \cite{Ma17jctb}]\label{q-jiema}
 	Is it true that $\nim_3(n;K_3)=t(n,5)$?
 \end{question}
The lower bound in Question~\ref{q-jiema} follows by taking a blow-up of a 2-edge-colouring of $K_5$
 without a monochromatic triangle, and assigning the third colour to all pairs inside a part.
 %%%%%%%%%%%%%%%%%%%%%%%%%%%%%%%%%%%%%%%%%%%%%%%%%%%%%%%
 
 \subsection{Non-bipartite case}

In order to state some of our results, we have to introduce the following variant of Ramsey number. Given a set $X$, denote by ${X\choose i}$ (resp.\ ${X\choose\le i}$), the set of all subsets of $X$ of size $i$ (resp.\ at most~$i$). Given two graphs $H$ and $G$, a (not necessarily injective) map $\phi: V(H)\rightarrow V(G)$ is a \emph{homomorphism} if it preserves all adjacencies, i.e.~$\phi(u)\phi(v)\in E(G)$ for every $uv\in E(H)$, and we say that $G$ is a \emph{homomorphic copy} of~$H$.

\begin{definition}\label{def-r*}
  	Given a sequence of graphs $\cH=(H_1,\ldots, H_k)$, denote by $r^*(H_1,\ldots,H_k)$ the maximum integer $r^*$ such that there exists a colouring $\xi: {[r^*]\choose \le 2}\to [k]$ such that
 	\begin{enumerate}[(P1)]
 		\item\label{item-r*1} the restriction of $\xi$ to ${[r^*]\choose 2}$ is \emph{$(H_1,\dots,H_k)$-homomorphic-free} (that is, for each $i\in[k]$ there is no edge-monochromatic  homomorphic copy of $H_i$ in the $i$-th colour);
 		\item\label{item-r*2} for every distinct $i,j\in [r^*]$ we have
 		$\xi(\{i,j\})\not=\xi(\{i\})$, that is, we forbid a pair having the same colour as one of its points.
 	\end{enumerate}
 	For any $r'\le r^*$, we will call a colouring $\xi:{[r']\choose \le 2}\rightarrow [k]$ \emph{feasible} (with respect to $(H_1,\ldots,H_k)$) if it satisfies both~(P\ref{item-r*1}) and~(P\ref{item-r*2}). We say that $(H_1,\ldots,H_k)$ is \emph{nice} if every feasible colouring $\xi:{[r^*(H_1,\ldots,H_k)]\choose \le 2}\rightarrow [k]$ assigns the same colour to all singletons.
\end{definition} 
 
Note that the colour assigned by $\xi$ to the empty set $\emptyset\in {[r^*]\choose \le 2}$ does not matter. Note also that when $k=2$, due to~(P2), a feasible colouring should use the same colour on all singletons. Thus, $r^*(H_1,H_2)=\max\{\chi(H_1),\chi(H_2)\}-1$. If we ignore (P2), then we obtain the following variant of Ramsey number that was introduced by Burr, Erd\H os and Lov\'asz~\cite{BurrErdosLovasz76}.
Let $r_{\mathrm{hom}}(H_1,\ldots,H_k)$ be the \emph{homomorphic-Ramsey number}, that is the maximum integer $r$ such that there exists an $(H_1,\ldots,H_k)$-homomorphic-free colouring $\xi: {[r]\choose 2}\to [k]$. We remark that for the homomorphic-Ramsey number, the colours of vertices do not play a role. When all $H_i$'s are cliques, this Ramsey variant reduces to the classical graph Ramsey problem:
  \begin{equation}\label{eq:r}
     r_{\mathrm{hom}}(K_{a_1},\ldots,K_{a_k})= R(a_1,\dots,a_k)-1.
  \end{equation}
  
There are some further relations to~$r^*$. For example, by assigning the same colour $i$ to all singletons and using the remaining $k-1$ colours on pairs, one can see that
 \begin{equation}\label{eq:LowerRHom}
 r^*(H_1,\dots,H_k)\ge \max_{i\in [k]}\,r_{\mathrm{hom}}(H_1,\ldots,H_{i-1},H_{i+1},\ldots,H_k).
 %,\quad \mbox{arbitrary graphs $H_1,\dots,H_k$}.
 \end{equation}
 If some $H_i$ is bipartite, then the problem of $r^*$ reduces to $r_{\mathrm{hom}}$. Indeed, as $K_2$ is a homomorphic copy of any bipartite graph, when some $H_i$ is bipartite, no feasible colouring $\xi$ can use colour $i$ on any pair. Consequently, we have equality in~\eqref{eq:LowerRHom}. This is one of the reasons why we restrict to non-bipartite $H_i$ in this section. 
 
 It would be interesting to know if~\eqref{eq:LowerRHom} can be strict. We conjecture that if all $H_i$'s are cliques then there is equality in~\eqref{eq:LowerRHom} and, furthermore, every optimal colouring uses the same colour on all singletons:
 
 \begin{conj}\label{conj-r*=r}
 	For any integers $3\le a_1\le \ldots\le a_k$, $(K_{a_1},\ldots,K_{a_k})$ is nice. In particular, $r^*(K_{a _1},\ldots,K_{a_k})=R(K_{a_2},\ldots,K_{a_k})-1$.
 \end{conj}

It is worth noting that not all $k$-tuples are nice. For example, it is easy to show that $r^*(C_5,C_5,C_5)=r_{\mathrm{hom}}(C_5,C_5)=4$, where $C_i$ denotes the cycle of length $i$, while Figure~\ref{fig-C_5} shows a feasible colouring of ${[4]\choose \le 2}$ assigning two different colours to singletons. 

\begin{figure}[!htb]
	\centering
	\includegraphics[scale=1.8]{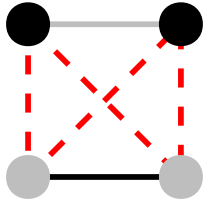}
	\caption{A feasible colouring of $K_4$ with respect to $(C_5,C_5,C_5)$, with two different colours on vertices.}
	\label{fig-C_5}
\end{figure}

 Our first result shows that this new variant plays a similar role for the function $\nim(\cdot)$ as the chromatic number in the Erd\H os-Simonovits-Stone Theorem~\cite{ErdosStone46,ErdosSimonovits66}.
 \begin{theorem}\label{thm-ram-nim}
 	Let $H_i$ be a non-bipartite graph, $i\in[k]$, and let
 	$r^*:=r^*(H_1,\ldots,H_k)$.
For every $\ep>0$,  we have that, for all large $n$,
 	\begin{equation}\label{eq:ram-nim}
 	\nim(n; H_1,\ldots,H_k)\le \left(1-\frac{1}{r^*}\right)\frac{n^2}{2}+ \ep n^2.
 	\end{equation}
 	Furthermore, if each $H_i$ is connected or there exists a feasible colouring of ${[r^*]\choose \le 2}$ with $k$ colours such that all singletons have the same colour, then we have $\nim(n; H_1,\ldots,H_k)\ge t(n,r^*)$.
 \end{theorem}

We also obtain the following stability result stating that if 
the number of NIM-edges is close to the bound in~\eqref{eq:ram-nim}, then the NIM-graph is close to a Tur\'an graph.  Let the \emph{edit distance} between  graphs $G$ and $H$ of the same order be
 \begin{equation}\label{eq:dedit}
 \dedit(G,H):=\min_{\sigma}\, |E(G)\bigtriangleup  \sigma(E(H))|,
 \end{equation}
 where the minimum is taken over all bijections $\sigma:V(H)\to V(G)$. In other words, $\dedit(G,H)$ is the minimum number of adjacency edits  needed to make $G$
and $H$ isomorphic.

 \begin{theorem}\label{thm-weakstab}
 	For any non-bipartite graphs $H_i$, $i\in[k]$, and any constant $\ep>0$, there exists $\de>0$ such that the following holds for sufficiently large $n$. If the number of \emph{\NIM}-edges of some  $\phi: {[n]\choose 2}\rightarrow[k]$ satisfies 
 	 $$
 	 \nim(\phi;H_1,\ldots,H_k)\ge\left(1-\frac{1}{r^*}\right)\frac{n^2}{2}-\de n^2,
 	 $$
 	  then $\dedit(G^{\nim},T(n,r^*))\le \ep n^2$, where  $r^*:=r^*(H_1,\ldots,H_k)$ and $G^{\nim}$ is the \emph{\NIM}-graph of $\phi$, i.e.,~the spanning subgraph with edge set \emph{$\NIM(\phi)$}.
 \end{theorem}

 Our next theorem shows that if Conjecture~\ref{conj-r*=r} is true, then this would determine the exact value of $\nim(\cdot)$ for a rather large family of graphs, including cliques. We call a graph $H$ \emph{homomorphism-critical} if it satisfies the following. If $F$ is a minimal  homomorphic copy of $H$, i.e.\ no proper subgraph of $F$ is a homomorphic copy of $H$, then for any edge $uv\in E(F)$, there exists a homomorphism $g:V(H)\rightarrow V(F)$ such that $|g^{-1}(u)|=|g^{-1}(v)|=1$, i.e.\ the pre-image sets of $u$ and $v$ are singletons. For example, all complete multipartite graphs with at least two parts of size 1 are homomorphism-critical. A simple consequence of this property is the following. As $F$ is minimal, it does not have any isolated vertices. Therefore, for any vertex $v\in V(F)$, there exists a homomorphism $g:V(H)\rightarrow V(F)$ such that $|g^{-1}(v)|=1$.
 
\begin{theorem}\label{thm-exact}
 	Let $(H_1,\ldots,H_k)$ be a nice sequence of non-bipartite graphs such that each $H_i$ is homomorphism-critical. Then for sufficiently large $n$,
 	\begin{eqnarray*}
 	\nim(n;H_{1},\ldots,H_{k})=t(n,r^*),
 	\end{eqnarray*}
 	where $r^*:=r^*(H_1,\ldots,H_k)$. Additionally, the \emph{\NIM}-graph of every extremal colouring is isomorphic to $T(n,r^*)$.
\end{theorem}

In the following theorems, we prove Conjecture~\ref{conj-r*=r} for $k=3$, and for $a_1=\ldots=a_4=3$ when $k=4$.  

\begin{theorem}\label{thm-conjfor3}
	For all integers $3\le a_1\le a_2\le a_3$, $(K_{a_1},K_{a_2},K_{a_3})$ is nice. In particular, $$r^*(K_{a_1},K_{a_2},K_{a_3})=R(a_2,a_3)-1.$$ 
\end{theorem}

\begin{theorem}\label{thm-4triangles}
	We have that $(K_3,K_3,K_3,K_3)$ is nice. In particular, $$r^*(K_3,K_3,K_3,K_3)=R(3,3,3)-1=16.$$ 
\end{theorem}

The following is an immediate corollary of Theorems~\ref{thm-exact},~\ref{thm-conjfor3} and~\ref{thm-4triangles}. In particular, the special case $a_1=a_2=a_3=3$ answers Question~\ref{q-jiema} affirmatively.
\begin{corollary}
	Let $3\le a_1\le a_2\le a_3$ be integers. Then for sufficiently large $n$, $$\nim(n;K_{a_1},K_{a_2},K_{a_3})=t(n,R(a_2,a_3)-1),$$
	$\nim_4(n;K_3)=t(n,16)$, and the \emph{$\NIM$}-graph of every extremal colouring is the corresponding Tur\'an graph.\qed
\end{corollary}

 %%%%%%%%%%%%%%%%%%%%%%%%%%%%%%%%%%%%%%%%%%%%%%%%%%%%%%%

\subsection{Bipartite graphs}

For bipartite graphs, we will provide a new family for which Problem~\ref{prob-KS} has a positive answer. Let us call an $h$-vertex graph $H$ \emph{weakly-reducible} if there exist $n_0\in \I N$ and a vertex $v\in V(H)$ such that $\ex(n,H-v)<\ex (n,H)-2^{2h^2}n$ for all $n\ge n_0$. (The function $2^{2h^2}$ comes from the proof and we make no attempt to optimise it.) Note that the family of weakly-reducible graphs includes all reducible graphs except the path of length $2$ and this inclusion is strict. For example, for integers $t>s\ge 2$, the disjoint union of the complete bipartite graphs $K_{2,t}$ and $K_{2,s}$ is weakly-reducible but not reducible; this
can be easily deduced from the result of F\"uredi~\cite{Furedi96jcta} that $\ex(n,K_{2,k})=(\sqrt{k}/2+o(1))\,n^{3/2}$ for any fixed $k\ge 2$ as $n\to\infty$.

\begin{theorem}\label{thm-weaklyred}
	Let $H$ be a weakly-reducible bipartite graph and $n$ be sufficiently large. Then
	\begin{eqnarray*}
	\nim_2(n;H)=\ex (n,H).
	\end{eqnarray*}
	Furthermore, every extremal colouring has one of its colour classes isomorphic to a maximum $H$-free graph of order~$n$.
\end{theorem} 

For a general bipartite graph $H$, we give in the following two theorems a weaker bound with an additive constant error term, namely, 
$$\nim_2(n;H)\le \ex(n,H)+O_H(1).$$ 
This provides more evidence towards Problem~\ref{prob-KS}.

\begin{theorem}\label{thm-cyclic}
	Let $H$ be a bipartite graph on at most $h$ vertices containing at least one cycle. Then for sufficiently large $n$,
	\begin{eqnarray*}
	\nim_2(n;H)\le\ex (n,H)+h^2.
	\end{eqnarray*}
\end{theorem}
For more than 2 colours, we obtain an asymptotic result for trees. Fix a tree $T$, by taking random overlays of $k-1$ copies of extremal $T$-free graphs, we see that $\nim_k(n;T)\ge (k-1)\ex(n,T)-O_{k,|T|}(1)$ (this construction is from Ma~\cite{Ma17jctb}). We prove that this lower bound is asymptotically true.
\begin{theorem}\label{thm-tree}
    Let $T$ be a forest with $h$ vertices. If $k=2$ or if $T$ is a tree, then there exists a constant $C:=C(k,h)$ such that, for all sufficiently large $n$,
	\begin{eqnarray*}
	\big|\,\nim_k(n;T)-(k-1)\,\ex (n,T)\,\big|\le C(k,h).
	\end{eqnarray*}
\end{theorem}

\nib{Organisation of the paper.} We first introduce some tools in Section~\ref{sec-prelim}. Then in Section~\ref{sec-bipartite}, we will prove Theorems~\ref{thm-weaklyred},~\ref{thm-cyclic}, and~\ref{thm-tree}. In Section~\ref{sec-ram-nim}, we will prove Theorems~\ref{thm-ram-nim} and~\ref{thm-weakstab}. We will present the proof for Theorem~\ref{thm-exact} in Section~\ref{sec-exact} and the proofs of Theorems~\ref{thm-conjfor3} and~\ref{thm-4triangles} in Section~\ref{sec-2specialcases}. Finally, in Section~\ref{sec-concluding} we give some concluding remarks.

%%%%%%%%%%%%%%%%%%%%%%%%%%%%%%%%%%%%%%%%%%%%%%%%%%%%%
\section{Preliminaries}\label{sec-prelim}

In this section, we recall and introduce some notation and tools. Recall that $[m]:=\{1,2,\ldots,m\}$ and ${X\choose i}$ (resp.\ ${X\choose \le i}$) denotes the set of all subsets of a set $X$ of size $i$ (resp.\ at most $i$). We also use the term \emph{$i$-set} for a set of size $i$.  We may abbreviate a singleton $\{x\}$ (resp.\ a pair $\{x,y\}$) as $x$ (resp.\ $xy$). If we claim, for example, that a result holds whenever $1\gg a\gg b>0$, this means that there are a constant $a_0\in (0,1)$ and a non-decreasing
 function $f : (0, 1) \rightarrow (0, 1)$ (that may depend on any previously defined constants or functions) such that the result holds for all $a,b\in (0,1)$ with $a\le a_0$ and $ b \leq f(a)$. 
   We may omit floors and ceilings when they are not essential.

Let $G=(V,E)$ be a graph. Its \emph{order} is $v(G):=|V|$ and its
\emph{size} is $e(G):=|E|$.
% and $H\subseteq G$ be a subgraph. 
 %the average degree of $G$  by $\bar d(G)$ and
The \emph{complement} of $G$ is $\overline{G}:=\left(V,{V\choose 2}\setminus E\right)$.
The chromatic number of $G$ is denoted by~$\chi(G)$.
For $U\subseteq V$,
let $G[U]:=(U,\{xy\in E: x,y\in U\})$ denote the subgraph of $G$ induced by $U$.
Also, denote
 \begin{eqnarray*}
 N_G(v,U)&:=&\{u\in U\mid uv\in E\},\\
 d_{G}(v,U)&:=&|N_{G}(v,U)|,
 \end{eqnarray*}
 and abbreviate $N_G(v):=N_G(v,V)$ and $d_G(v):=d_{G}(v,V)$.
 Let $\delta(G):=\min\{d_G(v):v\in V\}$
  denote the minimum degree of~$G$.

  Let $\mathcal{U}=\{U_1,U_2,\ldots,U_k\}$ be a collection of disjoint subsets of $V$.
 %and $|\mathcal{U}|=\sum_{i=1}^{k}|U_i|$.
  We write $G[U_1,\ldots,U_k]$ or $G[\mathcal{U}]$ for the multipartite subgraph of $G$ with vertex set $U:=\cup_{i\in[k]} U_i$ where we keep the \emph{cross-edges} of $G$ (i.e.\ edges that connect two parts); equivalently, we remove all edges from $G[U]$ that lie inside a part $U_i\in\mathcal{U}$. In these shorthands, we may omit $G$ whenever it is clear from the context, e.g.\ writing $[U_1,\dots,U_k]$ for $G[U_1,\dots,U_k]$. We say that $\mathcal{U}$ is a \emph{max-cut $k$-partition} of $G$ if $e(G[U_1,\ldots,U_k])$ is maximised over all $k$-partitions of $V(G)$.

For disjoint sets $V_1,\dots,V_t$ with $t\ge 2$, let $\K{V_1,\dots,V_t}$ denote the complete $t$-partite graph with parts $V_1,\dots,V_t$. Its isomorphism class is denoted by $K_{|V_1|,\dots,|V_t|}$. For example, if part sizes differ by at most $1$, then we get the Tur\'an graph
$T(|V_1|+\ldots+|V_t|,t)$. Let $M_{h}$ denote the matching with $h$ edges.

\begin{definition}
	For an edge-colouring $\phi:{[n]\choose 2}\rightarrow [k]$ of $G:=K_n$, define \emph{$\NIM(\phi;H_1,\ldots,H_k)$} to be the set of all edges not contained in any monochromatic copy of $H_i$ in colour $i$, and let \emph{$\nim(\phi;H_1,\ldots,H_k):=|\NIM(\phi;H_1,\ldots,H_k)|$}. Thus
	\begin{eqnarray*}
		\nim(n;H_1,\ldots,H_k)=\underset{\phi:E(K_n)\rightarrow[k]}{\max}\nim(\phi;H_1,\ldots,H_k).
	\end{eqnarray*}
	If the $H_i$'s are all the same graph $H$, then we will use the shorthands \emph{$\NIM_k(\phi;H)$}, $\nim_k(\phi;H)$ and $\nim_k(n;H)$ respectively. Also, we may drop $k$ when $k=2$
	and omit the graphs $H_i$ when these are understood. 
	Let $G^{\nim}$ be the spanning subgraph of $G$ with \emph{$E(G^{\nim})=\NIM(\phi;H_1,\ldots,H_k)$}. For $i\in [k]$, denote by $G_i$ and $G^{\nim}_i$ the spanning subgraphs of $G$ with edge-sets $E(G_i)=\{e\in E(G):\phi(e)=i\}$ and $E(G^{\nim}_i)=\{e\in E(G^{\nim}):\phi(e)=i\}$. We call an edge $e\in E(G^{\nim})$ (respectively, $e\in E(G^{\nim}_i)$) a \emph{\NIM-edge}  (resp.\ a \emph{\NIM-$i$-edge}). 
\end{definition}

\begin{definition}\label{def-bu}
 For $\xi:{[t]\choose \le 2}\rightarrow [k]$ and disjoint sets $V_1,\dots,V_t$, the \emph{blow-up colouring} $\xi(V_1,\dots,V_t): {V_1\cup \dots\cup V_t\choose 2}\to [k]$ is defined by
\begin{eqnarray*}
		\xi(V_1,\dots,V_t)(xy):=\begin{cases}
			\xi(ij),&\quad\text{if}\ xy\in E(\K{V_i,V_j}),\\
			\xi(i),&\quad\text{if}\ x,y\in V_i.
		\end{cases}
	\end{eqnarray*}
	If $|V_i|=N$ for every $i\in[t]$, then we say that $\xi(V_1,\dots,V_t)$ is an $N$-\emph{blow-up} of $\xi$.
\end{definition}

We say that a colouring $\phi$ \emph{contains} another colouring $\psi$ and denote this by $\phi\supseteq \psi$ if $\psi$ is a restriction of $\phi$. In particular, $\phi\supseteq \xi(V_1,\dots,V_t)$ means that $\phi$ is defined on every pair inside $V_1\cup\dots\cup V_t$ and coincides with $\xi(V_1,\dots,V_t)$ there.

We will make use of the multicolour version of the Szemer\'edi Regularity Lemma (see, for example,~\cite[Theorem 1.18]{KomlosSimonovits96}). Let us recall first the relevant definitions. Let $X,Y\subseteq V(G)$ be disjoint non-empty sets of vertices in a graph $G$. The \emph{density} of $(X,Y)$ is    
 $$d(X,Y):=\frac{e(G[X,Y])}{|X|\,|Y|}.
 $$
  For $\ep>0$, the pair $(X,Y)$ is {\it $\ep$-regular} if for every pair of subsets $X'\subseteq X$ and $Y'\subseteq Y$ with $|X'|\ge\ep|X|$ and $|Y'|\ge\ep|Y|$, we have $|d(X,Y)-d(X',Y')|\le \ep$. Additionally, if $d(X,Y)\ge\ga$, for some $\ga>0$, we say that $(X,Y)$ is {\it $(\ep,\ga)$-regular}.   A partition $\cP=\{V_1,\ldots,V_m\}$ of $V(G)$ is an {\it $\ep$-regular partition} of a $k$-edge-coloured graph $G$ if
 \begin{enumerate}
 \item for all $ij\in {[m]\choose 2}$, $\big|\,|V_i|-|V_j|\,\big|\le 1$; 
 \item  for all but at most $\ep {m\choose 2}$ choices of $ij\in {[m]\choose 2}$, the pair $(V_i,V_j)$ is $\ep$-regular in each colour.
 \end{enumerate}
 
\begin{lemma}[Multicolour Regularity Lemma]\label{lm:RL}
	For every real $\ep>0$ and integers $k\ge 1$ and $M$, there exists $M'$ such that every $k$-edge-coloured graph $G$ with $n\ge M$ vertices admits an $\ep$-regular partition $V(G)=V_1\cup \ldots\cup V_r$ with $M\le r\le M'$.
\end{lemma}

Given $\ep,\ga>0$, a graph $G$, a colouring $\phi:E(G)\rightarrow [k]$ and a  partition $V(G)=V_1\cup\dots\cup V_r$, define the \emph{reduced graph} 
 \begin{equation}\label{eq:R}
 R:=R(\ep,\gamma,\phi,(V_i)_{i=1}^r)
 \end{equation}
  as follows:
 $V(R)=\{V_1,\ldots,V_r\}$ and $V_i$ and $V_j$ are adjacent in $R$ if $(V_i,V_j)$ is $\ep$-regular with respect to the colour-$\ell$ subgraph of $G$ for every $\ell\in[k]$ and the colour-$m$ density of $(V_i,V_j)$ is at least $\gamma$ for some $m\in[k]$. For brevity, we may omit $\phi$ or $(V_i)_{i=1}^r$ in~\eqref{eq:R} when these are clear. 
 The graph $R$ comes with the \emph{majority edge-colouring} which assigns to  each edge $V_iV_j\in E(R)$ the colour that is the most common one among the edges in $G[V_i,V_j]$ under the colouring $\phi$. In particular, the majority colour has density at least $\gamma$ in $G[V_i,V_j]$. We will use the following consequence of the Embedding Lemma (see e.g.~\cite[Theorem 2.1]{KomlosSimonovits96}).

\begin{lemma}[Emdedding Lemma]\label{lem-embedding}
	Let $H$ and $R$ be graphs and let $1\ge \gamma\gg \ep\gg 1/m>0$. Let $G$ be a graph obtained by replacing every vertex of $R$ by $m$ vertices, and replacing the edges of $R$ with $\ep$-regular pairs of density at least $\gamma$. If $R$ contains a  homomorphic copy of $H$, then $H\subseteq G$. 
\end{lemma}

We will also need the Slicing Lemma (see e.g.~\cite[Fact 1.5]{KomlosSimonovits96}).

\begin{lemma}[Slicing Lemma]\label{lem-slicing}
	Let $\ep,\alpha,\gamma\in (0,1)$ satisfy $\ep\le \min\{\gamma,\alpha,1/2\}$. If $(A,B)$ is an $(\ep,\gamma)$-regular pair, then for any $A'\subseteq A$ and $B'\subseteq B$ with $|A'|\ge\al|A|$ and $|B'|\ge\al|B|$, we have that $(A',B')$ is an $(\ep', \gamma-\ep)$-regular pair, where $\ep':=\max\{\ep/\al,2\ep\}$.
\end{lemma}

\medskip\noindent\textbf{Conventions:}
%\meskip\noindent 
Throughout the rest of this paper, we will use $G$ as an edge-coloured $K_n$. For a given number of colours $k$ and a sequence of graphs $(H_i)_{i=1}^k$, we will always write $\psi: {[n]\choose 2}\rightarrow [k]$ for an extremal colouring realising $\nim(n;H_1,\ldots,H_k)$. We do not try to optimise the constants nor prove most general results, instead aiming for the clarify of exposition.

%We use the asymptotic notation, such as $O(1)$, to hide constant that do not depend on $n$.

\section{Proofs of Theorems~\ref{thm-weaklyred},~\ref{thm-cyclic} and~\ref{thm-tree}}\label{sec-bipartite}

By adding isolated vertices, we can assume that each graph $H_i$ has even order. The following proposition will be frequently used. It basically says that there are no monochromatic copies of $K_{v(H_i),v(H_i)/2}$ in colour $i$ that contains a $\NIM$-$i$-edge. Its proof follows from the fact that every edge of $K_{v(H_i),v(H_i)/2}$ is in an $H_i$-subgraph.

\begin{prop}\label{prop-completebipartite}
	For every graph $G$, fixed bipartite graphs $H_1,\ldots, H_k$, and a $k$-edge-colouring $\phi:E(G)\rightarrow [k]$, we have the following for every vertex $v\in V(G)$ and  $i\in [k]$. Let $U_i:=\{v'\in V(G):vv'\in G_{i}^\nim\}$. 
	\begin{enumerate}[(i)]
		\item For every vertex $u\in U_i$, the graph $G_{i}[N_{G_{i}}(v)\setminus \{u\},N_{G_{i}}(u)\setminus \{v\}]$ is $K_{v(H_i),v(H_i)/2}$-free.
		\item The graph $G_{i}[U_i,V\setminus (U_i\cup \{v\})]$ is $K_{v(H_i),v(H_i)/2}$-free.\qed
	\end{enumerate} 
\end{prop}

One of the key ingredients for the 2-colour case for bipartite graphs is the following lemma, which is proved by extending an averaging argument of Ma~\cite{Ma17jctb}. It states that any 2-edge-colouring of $K_n$ has only linearly many NIM-edges, or there is neither large NIM star nor matching in one of the  colours.

\begin{lemma}\label{lem-StMat}
	For any $h$-vertex bipartite $H$ with $h$ even and any $2$-edge-colouring $\phi$ of $G:=K_n$ with	 $\nim(\phi;H)> 2^{2h^2}n$,
	there exists $i\in[2]$ such that $G^{\nim}_i$ is $\{K_{1,h}, M_{h/2}\}$-free.
\end{lemma}
\begin{proof}
	We may assume, without loss of generality, that $G^{\nim}_1$ contains $K_{1,h}$, since otherwise 
	$$\nim(\phi;H)\le 2\cdot \ex(n,K_{1,h})\le (h-1)n,$$ 
	contradicting $\nim(\phi;H)> 2^{2h^2}n$. Let $S_v$ be an $h$-star in $G^{\nim}_1$ centred at $v$. We will show that if $G^{\nim}_2$ contains the star $K_{1,h}$ (Case 1) or the matching $M_{h/2}$ (Case 2), then it follows that $\nim(\phi;H)\le 2^{2h^2}n$, which is a contradiction. In each case, we will define a set  $S\subseteq V(G)$, with $h+1\le |S|\le h^2$, containing $S_v$ as follows. In Case 1, let $S_u$ be an $h$-star centred at $u$ in $G^{\nim}_2$ ($u$ and $v$ are not necessarily distinct). Define $S=V(S_v)\cup V(S_u)$ with $h+1\le |S|\le 2h+2$. In Case 2,  let $M\subseteq G^{\nim}_2$ be a matching with edge set $\{e_1,\ldots,e_{h/2}\}$, where $e_i=z_{i,1}z_{i,2}$ for every $1\le i\le h/2$. Denote $Z:=\cup_{i=1}^{h/2}\{z_{i,1},z_{i,2}\}$. For each edge $e_i\in E(M)$, without loss of generality, assume that $d_{G_2}(z_{i,1})\ge d_{G_2}(z_{i,2})$. Define iteratively for every $i=1,\ldots,h/2$ a set $U'_i$ as follows,
	\hide{	\begin{eqnarray*}
			\begin{cases}
				U'_i\subseteq N_{G_2}(z_{i,1})\setminus\left(Z\cup\left(\cup_{j=1}^{i-1}U'_j\right)\right),~|U'_i|=\frac{h}{2} &\quad\text{if }|N_{G_2}(z_{i,1})\setminus\left(Z\cup\left(\cup_{j=1}^{i-1}U'_j\right)\right)|\ge h/2,\\
				U'_i=N_{G_2}(z_{i,1})\setminus\left(Z\cup\left(\cup_{j=1}^{i-1}U'_j\right)\right) &\quad\text{otherwise.} \ 
			\end{cases}
	\end{eqnarray*}}
	\begin{eqnarray*}
		\begin{cases}
			U'_i\subseteq W_i,~|U'_i|=\frac{h}{2}, &\quad\text{if }|W_i|\ge h/2,\\
			U'_i=W_i, &\quad\text{otherwise,} \ 
		\end{cases}
	\end{eqnarray*}
	where $W_i:=N_{G_2}(z_{i,1})\setminus\left(Z\cup\left(\cup_{j=1}^{i-1}U'_j\right)\right)$; further define $U_i:=U'_i\cup\{z_{i,1},z_{i,2}\}$. Finally, set $S:=\left(\cup_{i=1}^{h/2}U_i\right)\cup V(S_v)$. So $h+1\le |S|\le h+1+(h/2+2)\cdot h/2\le h^2$.
	% and fix an ordering of the vertices in $S=\{s_1,\ldots,s_{|S|}\}$ such that $s_1,\ldots,s_h$ are the leaves of $S_v$.
	
	We now define a partition of $V(G)\setminus S$ that will be used in both Case 1 and Case 2. For each vertex $w\in V(G)\setminus S$, denote by $f_w$ the function $S\to [2]$ whose value on $s\in S$ is $f_w(s)=\phi(sw)$. In other words, ${f_w}$ encodes the colours of the edges from $w$ to $S$. Define
	%Now, we partition the vertices in $V(G)\setminus S$ according to their associated vectors:  
	\begin{eqnarray*}
		Y_1&:=&\{\,v\in V(G)\setminus S: |f_v^{-1}(2)|<h/2\,\},\\\
		Y_2&:=&\{\,v\in V(G)\setminus S: |f_v^{-1}(1)|<h/2\,\},\\
		X&:=& V(G)\setminus(S\cup Y_1\cup Y_2).\nonumber
	\end{eqnarray*}
	Thus $X$ consists of those $v\in V(G)\setminus S$ that send at least $h/2$ edges of each colour to $S$.

	We will show in the following claims that, for each class in this partition,  there are few vertices in that class or the number of NIM-edges incident to it is linear.
	
	\begin{claim}\label{cl-niminX}
		$e(G^{\nim}[X])\le h{|S| \choose h/2}n$.
	\end{claim}
	\begin{proof}[Proof of Claim.]
		Assume to the contrary that $e(G^{\nim}_i[X])\ge h{|S| \choose h/2}n/2$, for some $i\in [2]$. Then there exists a vertex $x\in X$ with $d_{G^{\nim}_i[X]}(x)\ge  h{|S|\choose h/2}$. By the definition of $X$, each vertex in $N_{G^{\nim}_i[X]}(x)$ has at least $h/2$ $G_i$-neighbours in $S$. By the Pigeonhole Principle, there exists a copy of $K_{h,h/2}\subseteq G_i[N_{G^{\nim}_i[X]}(x),S]$, which is a contradiction by Proposition~\ref{prop-completebipartite}(ii).  
	\end{proof}
	\begin{claim}\label{cl-verticesinY}
		$|Y_1|< h\cdot 2^{|S|}$.
	\end{claim}
	\begin{proof}[Proof of Claim.]
		Assume to the contrary that $|Y_1|\ge h\cdot 2^{|S|}$. Since the total number of functions $S\to [2]$ is $2^{|S|}$, by averaging, there exists a function $f$ and a subset $Y_{f}\subseteq Y_1$ with $|Y_{f}|\ge h$ such that for all vertices $y\in Y_{f}$, the functions $f$ and $f_y$ are the same. By the definition of $Y_1$, there is a subset $I\subseteq V(S_v)\setminus\{v\}$ with $|I|\ge h/2$ such that for all $s\in I$, $f(s)=1$, i.e., all pairs between $Y_{f}$ and $I$ are of colour 1. Recall that $S_v$ is the $h$-star consisting of \NIM-$i$-edges, thus, there exists a copy of $K_{h/2,h}\subseteq G_1[N_{G^{\nim}_1}(v),Y_{f}]$, which contradicts Proposition~\ref{prop-completebipartite}(ii).
	\end{proof} 
	
	We now show that $Y_2$ has also to be small (given that $G^{\nim}_2$ contains a large star or matching), otherwise $\nim(\phi,H)$ is linear.
	
	\medskip
	
	\nib{Case 1:} $G^{\nim}_2$ has the star $K_{1,h}$.\medskip
	
	\noindent	A similar argument as in Claim~\ref{cl-verticesinY} (with $S_u$ playing the role of $S_v$) shows that $|Y_2|<h\cdot 2^{|S|}$.
	
	\medskip
	
	\nib{Case 2:} $G^{\nim}_2$ has the matching $M_{h/2}$.\medskip 
	
	\noindent By the definition of $S$, all the $U_i$'s are pairwise disjoint and $h+1\le |S|\le h^2$,~see Figure~\ref{fig-starmatching}.
	\begin{figure}[!htb]
		\centering
		\includegraphics[scale=.8]{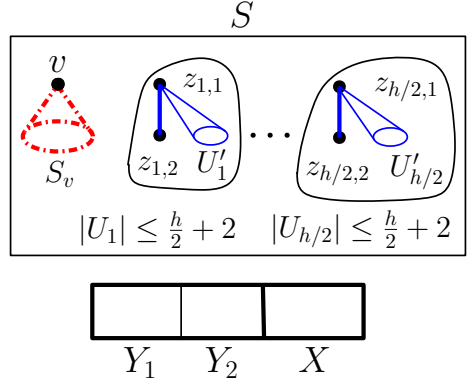}
		\caption{Case~$2$ of Lemma~\ref{lem-StMat}}
		\label{fig-starmatching}
	\end{figure} 	
	Suppose that $|Y_2|\ge h\cdot 2^{|S|}$. Again there exists a function $f:S\to [2]$ with $|f^{-1}(1)|<h/2$ and a subset $Y_{f}\subseteq Y_2$ with $|Y_{f}|\ge h$, such that for all vertices $y\in Y_{f}$, $f_y$ is the same as $f$. We will use the following claim.
	\begin{claim}\label{cl-onered}
		For every $1\le i\le h/2$, there exists $w\in U_i$ such that $f(w)=1$.
	\end{claim}
	\begin{proof}[Proof of Claim.]
		For a fixed $1\le i\le h/2$, assume to the contrary that for all $s\in U_i$, we have $f(s)=2$, i.e.,~$E(G[U_i,Y_{f}])\subseteq E(G_2)$. Thus, $|N_{G_2}(z_{i,1})\setminus(\cup_{j=1}^{i-1}U'_j\cup Z)|\ge |Y_{f}|$. Consequently, $|U'_i|=h/2$. Therefore, there exists $K_{h/2,h}\subseteq  G_2[U_i\setminus \{z_{i,1}\},Y_{f}]$, which contradicts Proposition~\ref{prop-completebipartite}(i) with $z_{i,1}$ and $z_{i,2}$ playing the roles of $v$ and $u$ respectively.
	\end{proof}
	By Claim~\ref{cl-onered} together with the fact that the $U_i$'s are pairwise disjoint, $f$ assumes value 1 at least $h/2$ times, which contradicts $Y_{f}\subseteq Y_2$. Therefore, in both cases, $|Y_2|<h\cdot 2^{|S|}$.

	Let $Y:=Y_1\cup Y_2$. Since $|S|\le h^2$, by Claims~\ref{cl-niminX} and~\ref{cl-verticesinY}, we get that
	\begin{eqnarray*}
		\nim(\phi;H)&\le& e(G^{\nim}[S])+e(G^{\nim}[S,V\setminus S])+e(G^{\nim}[Y])+e(G^{\nim}[Y,X])+e(G^{\nim}[X])\\
		&\le& |S|\cdot n+|Y|\cdot n+e(G^{\nim}[X])\\
		&\le&
		%(|S|+|Y|)n+e(G^{\nim}[X])\le 
		(h^2+2\cdot h 2^{h^2})n+h{h^2 \choose h/2}n\ <\ 2^{2h^2}n,
	\end{eqnarray*}
	a contradiction.
	
	This completes the proof of Lemma~\ref{lem-StMat}.
\end{proof}

%%%%%%%%%%%%%%%%%%%%%%%%%%%%%%%%%%%%%%%%%%%%%%%%%%%%%%%%%%%

\subsection{Weakly-reducible bipartite graphs}\label{sec-weaklyred}

\begin{proof}[Proof of Theorem~\ref{thm-weaklyred}]
	Let $H$ be a weakly-reducible bipartite graph. Let $h=v(H)$ and $w\in V(H)$ be a vertex such that $\ex(n,H-w)<\ex (n,H)-2^{2h^2}n$ for $n\ge n_0$. In particular, we have that
	$\ex(n,H)>2^{2h^2}n$ for $n\ge n_0$. Thus by Lemma~\ref{lem-StMat}, we may assume that there is $i\in [2]$ such that $e(G^{\nim}_i)\le \ex(n,\{K_{1,h},M_{h/2}\})\le h^2$. By the symmetry between the two colours, let us assume that $i=1$. Suppose that $E(G^{\nim}_1)\not=\emptyset$ as otherwise we are trivially done. We now distinguish the following two cases.
	
	\medskip
	
	\nib{Case 1:} For every edge $e=uv\in E(G^{\nim}_1)$, $d_{G_1}(u)\le 10h$ and $d_{G_1}(v)\le 10h$.\medskip
	
	\noindent In this case, pick one such edge, $e=uv$, and define $V_1=\left(N_{G_1}(u)\cup N_{G_1}(v)\right)\setminus\{u,v\}$. So $|V_1|\le d_{G_1}(u)+d_{G_1}(v)\le 20h$. Let     
	$$V_2:=V(G)\setminus(V_1\cup\{u,v\})=N_{G_2}(u)\cap N_{G_2}(v).$$ 
	Note that the subgraph of $G^{\nim}_2$ induced on vertex set $V_2$ satisfies $e(G^{\nim}_2[V_2])\le \ex(n,H-w)$. Otherwise, a copy of $H-w$ in $G^{\nim}_2[V_2]$ together with $u$ forms a copy of $H$ in colour $2$. Recall that $|V_1|\le 20h$, $e(G^{\nim}_1)\le h^2$ and $V_1\cup V_2\cup\{u,v\}$ is a partition of $V(G)$. Therefore for large $n$, we have
	\begin{eqnarray}%\label{eq-nimreducible}
	\nim(\psi;H)&\le& e(G^{\nim}_2)+e(G^{\nim}_1)\ \le\ e(G^{\nim}_2[V_2])+(|V_1|+2)n+h^2\\ 
	&\le& \ex(n, H-w)+30hn\nonumber\ \le\ \ex(n,H)-2^{2h^2}n+30hn\ <\ \ex(n,H).\nonumber
	\end{eqnarray}
	
	\noindent	\textbf{Case 2:} There exists an edge $e=uv\in E(G^{\nim}_1)$ such that $d_{G_1}(u)\ge 10h$.\medskip
	
	\noindent Pick $A\subseteq N_{G_1}(u)$ with $|A|=10h$, and denote 
	\begin{eqnarray}
	X&:=&\{z\in V(G)\setminus (A\cup \{u,v\}): d_{G_2}(z,A)\ge h\}\nonumber,\\
	Y&:=&\{z\in V(G)\setminus (A\cup\{u,v\}): d_{G_1}(z,A)\ge h\}\setminus X\nonumber.
	\end{eqnarray}
	Note that $X\cup Y\cup A\cup \{u,v\}$ is a partition of $V(G)$. We will use the following claims.
	\begin{claim}\label{cl-nimtoX}
		For every vertex $w\in X\cup Y$, $d_{G^{\nim}_2}(w,X)<h{10h\choose h}$.
	\end{claim}
	\begin{proof}[Proof of Claim.]
		Assume to the contrary that there exists a vertex $w\in X\cup Y$ with $d_{G^{\nim}_2}(w,X)\ge h{10h\choose h}$, and define $S:= N_{G^{\nim}_2}(w,X)$. Since $|A|=10h$ and vertices in $S$ all have $G_2$-degree at least $h$ in $A$, there exists a subset of $S$ of size at least $h$ such that its vertices are connected in $G_2$ to the same $h$ vertices in $A$, i.e.,~$K_{h,h}\subseteq G_2[S,A]$, which contradicts Proposition~\ref{prop-completebipartite}(ii).
	\end{proof}
	Define $Y'=Y\cap N_{G_1}(v)$ to be the set of all vertices in $Y$ that are adjacent to $v$ with a $1$-coloured edge, and $Y''=Y\setminus Y'$.
	\begin{claim}\label{cl-y'size}
		$|Y'|< {10h \choose h}h$.
	\end{claim}
	\begin{proof}[Proof of Claim.]
		Assume to the contrary that $|Y'|\ge {10h \choose h}h$. Since all vertices in $Y'$ have at least $h$ $G_1$-neighbours in $A$, there exists a copy of $K_{h,h}\subseteq G_1[Y',A]$, which extends to a copy of $K_{h+1,h+1}\supseteq H$ containing the edge $uv\in E(G^{\nim}_1)$, a contradiction.  
	\end{proof}
	
	By Claims~\ref{cl-nimtoX},~\ref{cl-y'size} and since $|A|=10h$, the number of edges in $G^{\nim}_2$ with at least one end point in the set $A\cup Y'\cup X\cup\{u,v\}$ is at most $3h{10h\choose h}n$. It remains to estimate $e(G^{\nim}_2[Y''])$. We claim that $e(G^{\nim}_2[Y''])\le\ex(n,H-w)$. Otherwise, since all the edges connecting $v$ to $Y''$ have colour 2, we can extend the copy of $H-w\subseteq G^{\nim}_2[Y'']$ to a copy of $H$ by adding~$v$. This contradicts the definition of $G^{\nim}_2$. Hence,
	\begin{eqnarray}
	\nim(\psi;H)&=&e(G^{\nim}_2)+e(G^{\nim}_1)\le\ 3h{10h\choose h}n+\ex(n,H-w)+h^2\nonumber\\
	&<&\ex(n,H)-2^{2h^2}n+4h{10h\choose h}n\ <\ \ex(n,H).\nonumber 
	\end{eqnarray}
	Thus, any colouring with NIM-edges of two different colours is not extremal. 
\end{proof}
%%%%%%%%%%%%%%%%%%%%%%%%%%%%%%%%%%%%%%%%%%%%%%%%%%%%%%%%%%%
\subsection{General bipartite graphs}\label{sec-cyclictree}
In this subsection, we will prove Theorems~\ref{thm-cyclic} and~\ref{thm-tree}. 
\begin{proof}[Proof of Theorem~\ref{thm-cyclic}]
	As $H$ contains a cycle, $\ex(n,H)/n\rightarrow \infty$ as $n\rightarrow\infty$. Then by Lemma~\ref{lem-StMat}, we may assume that, for example, $G^{\nim}_1$ is $\{K_{1,h},M_{h/2}\}$-free. Since $G^{\nim}_2$ is $H$-free, we immediately get that 
	\begin{eqnarray*}
		\nim_2(n,H)\le \ex(n,H)+\ex(n,\{K_{1,h},M_{h/2}\})\le \ex(n,H)+h^2,
	\end{eqnarray*}
	as desired.
\end{proof}

\begin{proof}[Proof of Theorem~\ref{thm-tree}] Let us first present the part of the proof which works for an arbitrary number of colours $k$ and any forest $T$. Let $h=v(T)$. 
	
	The stated lower bound on $\nim_k(n;T)$ can be obtained by using the argument of Ma~\cite{Ma17jctb}. Fix some maximum $T$-free graph $H$ on $[n]$ and take uniform independent permutations $\sigma_1,\dots,\sigma_{k-1}$ of $[n]$. Iteratively, for $i=1,\dots,k-1$, let the colour-$i$ graph $G_i$ consists of those pairs $\{\sigma_i(x),\sigma_i(y)\}$, $xy\in E(H)$, that are still uncoloured. Finally, colour all remaining edges with colour $k$. Clearly, all edges of colours between $1$ and $k-1$ are \NIM-edges. Since $e(H)\le hn=O(n)$, the expected size of $\sum_{i=1}^{k-1} e(G_i)$ is at least 
	$$
	(k-1)e(H)-{k-1\choose 2} e(H)^2{n\choose 2}^{-1}\ge(k-1)\ex(n,T)-k^2h^2.
	$$
	By choosing the permutations for which $\sum_{i=1}^{k-1} e(G_i)$ is at least its expectation, we obtain the required bound.

	Let us turn to the upper bound. Fix an extremal $G$ with colouring $\phi: {[n]\choose 2}\rightarrow [k]$, so $\nim_k(\phi;T)=\nim_k(n;T)$. For every $1\le i\le k$, denote
	\begin{eqnarray}
	A_i:=\{v\in V(G): \exists u\in V(G),~ uv\in E(G^{\nim}_i) \}\quad\text{and}\quad a_i:=|A_i|.\nonumber
	\end{eqnarray}
	In other words, $A_i$ is the set of all vertices incident with at least one $i$-coloured NIM-edge. Note that
	\begin{eqnarray}\label{eq-nimtree}
	\nim_k(\phi;T)\le \sum_{i=1}^{k}\ex(a_i,T).
	%\le \ex\left(\sum_{i=1}^{k}a_i,T\right).
	\end{eqnarray}
	Also, for every $X\subseteq [k]$, define
	\begin{eqnarray*}
		B_X:= \{v\in V(G): v\in A_i\Leftrightarrow i\in X \}=\cap_{i\in X}A_i\setminus(\cup_{j\notin X}A_j)\quad\text{and}\quad b_X:=|B_X|.
	\end{eqnarray*}
	In other words, $B_X$ is the set of vertices which are incident with edges in $G_i^{\nim}$ if and only if $i\in X$.
	By definition, for two distinct subsets $X,Y\subseteq [k]$, $B_X\cap B_{Y}=\emptyset$. 
	%Therefore, we have $n=\sum_{X\subseteq [k]}b_X$,  $\sum_{i=1}^{k}a_i=\sum_{X\subseteq [k]}|X|b_X$, and 
	
	\begin{claim}\label{cl-complements}
		For every two subsets $X,Y\subseteq [k]$ with $X\cup Y=[k]$, $\min\{b_X,b_Y\}<6kh$.
	\end{claim} 
	\begin{proof}[Proof of Claim.]
		Assume on the contrary that there exist two subsets $X,Y\subseteq [k]$ such that $X\cup Y=[k]$ and $b_X,b_Y\ge 6kh$. Let $B_X'\subseteq B_X$ and $B_Y'\subseteq B_Y$ be such that $|B_X'|=|B_Y'|=6kh$. By averaging, some colour, say colour 1, contains at least $1/k$ proportion of edges in $G[B_X',B_Y']$. Set $F=G_1[B_X',B_Y']$. Then there exists $F'\subseteq F$ on vertex set $B_X''\cup B_Y''$, where $B_X''\subseteq B_X'$ and $B_Y''\subseteq B_Y'$, such that the minimum degree of $F'$ is at least half of the average degree of $F$, that is,
		\begin{eqnarray*}
			\de(F')\ge \frac{e(F)}{|V(F)|}\ge \frac{|B_X'|\cdot |B_Y'|}{k\cdot (|B_X'|+|B_Y'|)}=\frac{(6kh)^2}{k\cdot 12kh}= 3h.
		\end{eqnarray*}
		Let $v\in V(T)$ be a leaf, $u$ be its only neighbour, and $T':=T-v$, where $T-v$ is the forest obtained from deleting the leaf $v$ from $T$. 
		Since $X\cup Y=[k]$, without loss of generality, we can assume that $1\in X$. Fix an arbitrary vertex $x\in B_X''$ and let $w$ be a $G^{\nim}_1$-neighbour of~$x$. (Such a vertex exists as $x\in B_X''\subseteq B_X$ and $1\in X$.) Then $\de(F'-w)\ge \de(F')-1\ge 2h$. We can then greedily embed $T'$ in $F'-w$ with $x$ playing the role of $u$. As this copy of $T'$ is in $F'-w\subseteq G_1$, together with $xw\in G^{\nim}_1$, we get a monochromatic copy of $T$ with an edge in $G^{\nim}$, a contradiction (see Figure~\ref{fig-tree}).
	\end{proof}
	\begin{figure}[!htb]
		\centering
		\includegraphics[scale=1.3]{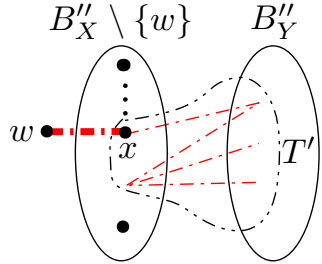}
		\caption{Finding a copy of $T\in G_1$.}
		\label{fig-tree}
	\end{figure}
	
	We will divide the rest of the proof into two cases.
	\medskip
	
	\nib{Case 1:} There exists a subset $X\subset [k]$ such that $|X|=k-1$ and $b_X\ge 6kh $.\medskip 
	
	\noindent Let $\{j\}=[k]\setminus X$, and $\cY$ be the collection of all subsets of $[k]$ containing $j$. By Claim~\ref{cl-complements}, $b_{Y}<6kh$, for every set $Y\in \cY$, implying that $a_j=\sum_{Y\in \cY}b_Y<2^k\cdot 6kh$. Hence, by~\eqref{eq-nimtree},
	\begin{eqnarray*}
		\nim_k(\phi;T)&\le&\sum_{i\in [k]}\ex(a_i,T)\ \le\ \sum_{i\in [k]\setminus\{j\}}\ex(a_i,T)+\ex(2^k\cdot 6kh,T)\\
		&\le& (k-1)\ex(n,T)+2^k\cdot 6kh^2.
	\end{eqnarray*}
	Thus the theorem holds in this case.
	\medskip
	
	\nib{Case 2:} For all subsets $X\subset [k]$ with $|X|=k-1$, we have $b_X< 6kh$.\medskip 
	
	\noindent By Claim~\ref{cl-complements}, we have $b_{[k]}\le 2\cdot 6kh$. Hence, all but at most $(k+2)6kh$ vertices are adjacent to \NIM-edges with at most $k-2$ different colours, which implies that they are in at most $k-2$ different sets $A_i$. Therefore, 
	\begin{equation}\label{eq:SumAi}
	a_1+\dots+a_k\le (k-2)n+12(k+2)kh.
	\end{equation}
	%by~\eqref{eq-nimtree},
	%\begin{eqnarray}
	%\nim(\phi;T)&\le& \ex\left(\sum_{i\in %[k]}a_i,T\right)\le\ex\left((k-2)n+12(k+2)kh,T\right)\nonumber
	%\end{eqnarray}
	
	Now our analysis splits further, depending on the cases of Theorem~\ref{thm-tree}. If $k=2$, then we are done by~\eqref{eq-nimtree} and~\eqref{eq:SumAi}:
	$$\nim_2(\phi;T)\le h(a_1+a_2)\le 96h^2\le \ex(n,T).
	$$
	
	Thus it remains to consider the case when $k\ge 3$ and $T$ is a tree. By taking the disjoint union of two maximum $T$-free graphs, we see that the Tur\'an function of $T$ is superadditive, that is, 
	\begin{equation}\label{eq:SubAdd}
	\ex(\ell,T)+\ex(m,T)\le \ex(\ell+m,T),\quad\mbox{for any $\ell,m\in \I N$}.
	\end{equation}
	The Fekete Lemma implies that $\ex(m,T)/m$ tends to a limit $\tau$. Since, for example, $\ex(m,T)\le hm$, we have that $\tau\le h$, in particular, $\tau$ is finite. Also, excluding the case $T=K_2$ when the theorem trivially holds,
	we have $\tau>0$. In particular, $|\,\ex(m,T)/m-\tau\,|< c$ for all large $m$, where $c:=\tau/(3k-4)>0$ satisfies $(\tau+c)(k-2)=(k-1)(\tau-2c)$. 
	
	%As any graph with minimum degree at least $h$ contains $T$, by repeatedly removing vertices of degree at most $h-1$ from an extremal graph for $T$, 
	%we have that $\ex(\ell +m,T)\le \ex(\ell,T)+hm$ for all $\ell,m\in\I N$.
	Thus~\eqref{eq-nimtree},~\eqref{eq:SumAi},~\eqref{eq:SubAdd} and the fact that $n$ is sufficiently large give that 
	\begin{eqnarray*}
		\nim_k(\phi;T)&\le& \ex(a_1+\dots+a_k,T)\ \le\ \ex((k-2)n+12(k+2)kh,T)\\
		%\ex((k-2)n+12(k+2)kh,T)\\
		&\le& (\tau+c) \big((k-2)n+12(k+2)kh\big)\ \le \ (k-1)(\tau-2c)n + 24(k+2)kh^2\\
		&\le& (k-1)(\tau-c)n\le (k-1)\ex(n,T).
	\end{eqnarray*}
	This finishes the proof of Theorem~\ref{thm-tree}.\end{proof}

%%%%%%%%%%%%%%%%%%%%%%%%%%%%%%%%%%%%%%%%%%%%%%%%%%%%%%
%%%%%%%%%%%%%%%%%%%%%%%%%%%%%%%%%%%%%%%%%%%%%%%%%%%%%%

\section{Proofs of Theorems~\ref{thm-ram-nim} and~\ref{thm-weakstab}}\label{sec-ram-nim}

%Fix integers $a_1,\ldots, a_k\ge 3$, throughout this section let $r^*=r^*(a_1,\ldots,a_k)$.

We need the following lemma, which states that the reduced graph of the \NIM-graph cannot have a large clique, linking the nim function to the new Ramsey variant $r^*$.

\begin{lemma}\label{lem-no-large-clique}
	For $i\in [k]$, let $H_i$ be a non-bipartite graph, and let $1/k,1/r\ge \gamma\gg \ep\gg 1/N>0$, where $r:=R(a_1-1,\ldots,a_k-1)$ and $a_i:=\chi(H_i)$. Let $V_1,\dots,V_m$ be disjoint sets, each of size at least $N$.
   Take any $\phi: {V\choose 2}\rightarrow [k]$, where $V:=V_1\cup\dots\cup V_m$,  and let $G^{\nim}$ be the \emph{\NIM}-graph of $\phi$. Then the graph $R:=R(\ep,\gamma,\phi|_{E(G^{\nim})},(V_i)_{i=1}^m)$ is $K_{r^*+1}$-free, where $r^*:=r^*(H_1,\ldots,H_k)$.
\end{lemma}
\begin{proof}
	Given the graphs $H_i$ with $a_i=\chi(H_i)$, and $r=R(a_1,\ldots,a_k)$, choose additional constants so that the following hierarchy holds:
	$$
	\frac{1}{r}\gg \gamma\gg\ep_1\gg\frac{1}{M}\gg\ep\gg \frac{1}{N}>0.$$
	Let the $V_i$'s and $\phi$ be as in the statement of the lemma. 
	%Let $V(R)=\{V_1,\ldots,V_m\}$. 
	For each $i\in[m]$, apply the Multicolour Regularity Lemma (Lemma~\ref{lm:RL}) with constants $\ep_1$ and $1/\ep_1$ to the $k$-coloured complete graph on $V_i$ to obtain an $\ep_1$-regular partition $V_i=U_{i,1}\cup\dots\cup U_{i,m_i}$ with $1/\ep_1\le m_i\le M$. 
	%The lemma applies since $|V_i|\ge N\gg M$. 
	Let $R_i:=R(\ep_1,\gamma,\phi|_{{V_i\choose 2}},(U_{i,j})_{j=1}^{m_i})$ be the associated reduced graph. 
	
Note that the fraction of the elements $xy\in {V_i\choose 2}$ with $x\in U_{i,a}$ and $y\in U_{i,b}$ such that the pair $(U_{i,a},U_{i,b})$ 
%has total density less than $k\gamma$, or 
 is not $\ep_1$-regular in some colour or satisfies $a=b$ is at most  $\ep_1+1/m_i$. Since $\gamma\le 1/k$, the remaining elements of ${V_i\choose 2}$ come from edges of~$R_i$. Recall that $m_i=v(R_i)$. Thus, we have  that 
 \begin{equation}\label{eq:eRi}
 e(R_i)\ge \frac{(1-\ep_1-1/m_i){|V_i|\choose 2}}{\lceil\, |V_i|/m_i\,\rceil^2}\ge (1-2\ep_1){m_i\choose 2}.
 \end{equation}
	
	Let $\xi:E(R)\to [k]$ be the colouring of $R$. We extend it to the vertices of $R$ as follows. Take $i\in [m]$. Let $\xi_i: E(R_i)\rightarrow [k]$ be the colouring of $R_i$. 
By~\eqref{eq:eRi} and since $v(R_i)\ge 1/\ep_1$ and $\ep_1\ll 1/r$, we have that $e(R_i)>\ex(m_i,K_{r})$. By Tur\'an's theorem~\cite{Turan41}, the graph $R_i$ contains an $r$-clique. By the definition of $r$, the restriction of the $k$-edge-colouring $\xi_i$ to this $r$-clique contains a colour-$p$ copy of $K_{a_p-1}$ for some $p\in[k]$. Let $\xi$ assign the colour $p$ to $V_i$.
	
	Suppose to the contrary that some $(r^*+1)$-set $A$ spans a clique in~$R$. The restriction of $\xi$ to  ${A\choose \le 2}$ violates either~(P\ref{item-r*1}) or~(P\ref{item-r*2}) from Definition~\ref{def-r*}. We will derive contradictions in both cases, thus finishing the proof. 
	If $\xi$ contains an edge-monochromatic  homomorphic copy of some $H_i$ in colour $i\in[k]$, then by
	the Embedding Lemma (Lemma~\ref{lem-embedding}) the colour-$i$ subgraph of $G^{\nim}$ contains a copy of $H_i$, a contradiction to $G^{\nim}$ consisting of the \NIM-edges. So suppose that~(P\ref{item-r*2}) fails, say, some pair $V_iV_j\in {A\choose 2}$ satisfies $\xi(V_iV_j)=\xi(V_i)$, call this colour $p$. By the definition of $\xi(V_i)$, $R_i$ contains an $(a_p-1)$-clique of colour $p$ under $\xi_i$, say with vertices $U_1,\dots,U_{a_p-1}\in V(R_i)$. Observe that $\ep_1\ge \max\{2\ep,\ep M\}\ge\max\{2\ep,\ep \cdot v(R_i)\}$ and 
	$p$ is the majority colour on edges in $G^{\nim}[V_i,V_j]$. The Slicing Lemma (Lemma~\ref{lem-slicing}) with e.g.\ $\al:=1/M$ gives that each pair $(V_j,U_h)$ with $h\in [a_p-1]$ is $(\ep_1,\gamma/2)$-regular in $G^{\nim}_p$. The Embedding Lemma (Lemma~\ref{lem-embedding}) gives a copy of $H_{p}$ in $G$ containing at least one (in fact, at least $\delta(H_p)$) edges of $G^{\nim}[V_i,V_j]$, a contradiction. 
\end{proof}

\begin{proof}[Proof of Theorem~\ref{thm-ram-nim}]
	For the upper bound, let $1\gg \ep\gg\gamma\gg\ep_1>0$. Let $n$ be large and
 suppose to the contrary that there exists some colouring $\phi: E(K_n)\rightarrow [k]$ that violates~\eqref{eq:ram-nim}. Apply the Multicolour Regularity Lemma (Lemma~\ref{lm:RL}) to the NIM-graph $G^{\nim}$ of $\phi$ with parameters $\ep_1$ and  $1/\ep_1$. A calculation similar to the one in~\eqref{eq:eRi} applies here, where additionally one has to discard at most $k\gamma {n\choose 2}$ edges in $\NIM(\phi)$ coming from pairs that have density less than $\gamma$ in each colour. By $\gamma\ll \ep$, we conclude
 that the reduced graph $R=R(\ep_1,\gamma,\phi|_{E(G^{\nim})})$ of $G^{\nim}$ has at least $(1-1/{r^*}+\ep/2)\frac{v(R)^2}{2}$ edges. By Tur\'an's theorem, $K_{r^*+1}\subseteq R$, contradicting Lemma~\ref{lem-no-large-clique}.

For the lower bound, take a feasible $k$-colouring $\xi$ of ${[r^*]\choose \le 2}$, where $r^*:=r^*(H_1,\dots,H_k)$. If possible, among all such colourings take one such that all singletons have the same colour. Consider the blow-up colouring $\phi:=\xi(X_1,\dots,X_{r^*})$ where the sets $X_i$ form an equipartition of~$[n]$. 

Let us show that every edge of $\K{X_1,\dots,X_{r*}}$ is a \NIM-edge. 
Take any copy $F$ of $H_i$ which is $i$-monochromatic in $\phi$. 
Since the restriction of $\xi$ to ${[r^*]\choose 2}$ has no  homomorphic copy of $H_i$ by~(P\ref{item-r*1}), the graph $F$ must use at least one edge that is inside some $V_j$. If $\xi$ assigns the value $i$ only to singletons, then no edge of the colour-$i$ graph $F$ can be a cross-edge. Otherwise, if $F$ is connected, then  $E(F)\subseteq {V_j\choose 2}$ because no edge between $V_j$ and its complement can have $\phi$-colour $i$  by~(P\ref{item-r*2}). We conclude that every cross-edge is a \NIM-edge, giving the required lower bound.\end{proof}

\begin{proof}[Proof of Theorem~\ref{thm-weakstab}]
	Choose $\ep\gg \ep_1\gg\de\gg\gamma\gg\ep_2\gg 1/n>0$. Let $\phi$ be as in the theorem. Apply the Regularity Lemma (Lemma~\ref{lm:RL})  to \NIM-graph $G^{\nim}$ with parameters $\ep_2$ and $1/\ep_2$ to get an $\ep_2$-regular partition $V(G^{\nim})=V_1\cup\dots\cup V_m$. Let $R=R(\ep_2,\ga,\phi|_{E(G^{\nim})},(V_i)_{i=1}^m)$ be the reduced graph. A similar calculation as in~\eqref{eq:eRi}  yields that $e(R)\ge\left(1-\frac{1}{r^*}-2\de\right)\frac{m^2}{2}$. On the other hand, by Lemma~\ref{lem-no-large-clique}, $R$ is $K_{r^*+1}$-free. Thus, the Erd\Ho s-Simonovits Stability Theorem~\cite{Erdos67a,Simonovits68} implies that $\dedit(R,  T(m,r^*))\le \ep_1 m^2/2$.
	Let a partition $V(R)=\cU_1\cup\dots\cup\cU_{r^*}$ minimise $|E(R)\bigtriangleup  E(\K{\cU_1,\dots,\cU_{r^*}})|$.
	We know that the minimum is at most $\ep_1 m^2/2$. Let $V(G^{\nim})=W_1\cup\dots\cup W_{r^*}$ be the partition induced by $\cU_i$'s, i.e., $W_i:=\cup_{V_j\in \cU_i} V_j$ for~$i\in[r^*]$. Let $G'$ be the graph obtained from $G^{\nim}$ by removing all edges that lie in any cluster $V_i$; or between those parts $V_i$ and $V_j$ such that $V_iV_j$ is not an $\ep_2$-regular pair or belongs to $E(R)\bigtriangleup  E(\K{\cU_1,\dots,\cU_{r^*}})$. We have
	$$
	 |  E(G^{\nim})\bigtriangleup E(G')|=|E(G^{\nim})\setminus E(G')|\le m\cdot\frac{(n/m)^2}{2}+\ep_2m^2\cdot\frac{n^2}{m^2}+|E(R)\bigtriangleup  E(T(m,r^*))|\cdot\frac{n^2}{m^2}\le \ep_1 n^2.$$
	As $e(G^{\nim})\ge (1-1/r^*)n^2/2-\de n^2$, we  have $e(G')\ge (1-1/r^*)n^2/2-2\ep_1 n^2$. Since $G'$ is $r^*$-partite (with parts $W_1,\ldots,W_{r^*}$), a direct calculation %(or the Erd\Ho s-Simonovits Stability Theorem) 
	gives that $\dedit(G', T(n,r^*))\le \ep n^2/2$.
	Finally, we obtain
	$$
	\dedit(G^{\nim},T(n,r^*))\le |E(G^{\nim})\bigtriangleup  E(G')|+\dedit(G',T(n,r^*))\le \ep_1 n^2+\frac{\ep n^2}{2}\le \ep n^2,$$
	as desired.
\end{proof}
%%%%%%%%%%%%%%%%%%%%%%%%%%%%%%%%%%%%%%%%%%%%%%%%%%%%%
\section{Proof of Theorem~\ref{thm-exact}}\label{sec-exact}

The following lemma will be useful in the forthcoming proof of Theorem~\ref{thm-exact}. It is proved by an easy modification of the standard proof of Ramsey's theorem.

\begin{lemma}[Partite Ramsey Lemma]\label{lm:PRL}
 For every triple of integers $k,r,u\in\I N$ there is $\rho=\rho(k,r,u)$ such that if $\phi$ is a $k$-edge-colouring of the complete graph on $Y_1\cup \dots\cup Y_r$, where $Y_1,\dots,Y_r$ are disjoint $\rho$-sets, then there are $u$-sets $U_i\subseteq Y_i$, $i\in [r]$, and $\xi:{[r]\choose \le 2}\to [k]$ such that $\xi(U_1,\dots,U_r)\subseteq \phi$. (In other words, we require that each ${U_i\choose 2}$ 
 and each bipartite graph $[U_i,U_j]$ is monochromatic.) 
\end{lemma}
 \begin{proof} We use induction on $r$ with the case $r=1$ being the classical Ramsey theorem.
Let $r\ge 2$ and set $N:=(u-1)k^r+1$. We claim that $\rho:=(2k)^N\,\rho(k,r-1,u)$ suffices here. Let $\xi$ be an arbitrary $k$-edge-colouring of the complete graph on $Y_1\cup \dots\cup Y_r$ where each $|Y_i|= \rho$. 

Informally speaking, we iteratively pick vertices $x_1,\dots,x_N$ in $Y_r$ shrinking the parts so that each new vertex $x_i$ is monochromatic to each part. Namely, we initially let $U_i^0:=Y_i$ for $i\in [r]$. Then for $i=1,\dots,N$ we repeat the following step. Given vertices $x_1,\dots,x_{i-1}$ and sets $U_r^{i-1}\subseteq Y_r\setminus\{x_1,\dots,x_{i-1}\}$ and $U_j^{i-1}\subseteq Y_j$ for $j\in [r-1]$, we let $x_i$ be an arbitrary vertex of $U_r^{i-1}$ and, for $j\in [r]$, let $U_j^{i}$ be a maximum subset of $U_j^{i-1}$ such that all pairs between $x_i$ and $U_j^{i}$ have the same colour, which we denote by $c_{j}^i\in[k]$. Clearly, $|U_j^{i}|\ge (|U_j^{i-1}|-1)/k$ (the $-1$ term is need for $j=r$), which is at least $(2k)^{N-i}$ by a simple induction on~$i$. Thus we can carry out all $N$ steps. Moreover, each of the the final sets $U_1^{N},\dots,U_{r-1}^N$ has size at least $\rho /(2k)^N= \rho(k,r-1,u)$. By the 
induction assumption, we can find $u$-sets $U_j\subseteq U_j^N$, $j\in [r-1]$, and $\xi:{[r-1]\choose \le 2}\to [k]$ with $\xi(U_1,\dots,U_{r-1})\subseteq \phi$.

Each selected vertex $x_i$ comes with a colour sequence $(c^i_1,\dots,c^i_r)\in [k]^r$. So we can find a set $U_r\subseteq \{x_1,\dots,x_N\}$ of $\lceil N/k^r\rceil=u$ vertices that have the same colour sequence $(c_1,\dots,c_r)$. Clearly, all pairs in ${U_r\choose 2}$ (resp.\ $[U_r,U_j]$ for $j\in [r-1]$) have the same colour $c_r$ (resp.\ $c_j$). Thus if we extend the colouring $\xi$ to ${[r]\choose \le 2}$ by letting $\xi(i,r):=c_i$ for $i\in [r-1]$ and $\xi(r):=c_r$, then $\xi(U_1,\dots,U_r)\subseteq \phi$, as required.
\end{proof}

The main step in proving Theorem~\ref{thm-exact} is given by the following lemma.

\begin{lemma}\label{lem-thm-exact} Under the assumptions of Theorem~\ref{thm-exact}, there is $n_0$ such that if $\phi$ is an arbitrary $k$-edge-colouring of $G:=K_n$ with $n\ge n_0$, $e(G^{\nim})\ge t(n,r^*)$ and 
	\begin{equation}\label{eq-minnimdeg}
	\de(G^{\nim})\ge \de(T(n,r^*)),
	\end{equation}
	where $\de$ denotes the minimum degree, then $G^\nim\cong T(n,r^*)$ (in particular, $e(G^{\nim})= t(n,r^*)$).
\end{lemma}

\begin{proof}
	Let $H_i$, $i\in[k]$, and $r^*$ be as in Theorem~\ref{thm-exact}. So $r^*=r^*(H_1,\ldots,H_k)$. Let
	 $$
	  N:=\max_{i\in[k]} v(H_i)-1\ge 1\quad\mbox{and}\quad 1\gg \ep\gg\ep_1\gg 1/n_0>0.
	  $$ 
	  Let $n\ge n_0$ and let $\phi$ be an arbitrary $k$-edge-colouring of $G:=K_n$. 
	
	Let $\mathcal{P}=\{V_1,\ldots,V_{r^*}\}$ be a max-cut $r^*$-partition of $G^\nim$. In particular,~for every $i,j\in [r^*]$ and every $v\in V_i$, we have $d_{G^{\nim}}(v,V_{j})\ge d_{G^{\nim}}(v,V_i)$. By applying Theorem~\ref{thm-weakstab} to $G^{\nim}$, we have 
	\begin{equation}\label{eq-p-cr}
	e(G^{\nim}[\mathcal{P}])\ge t(n,r^*)-\ep_1n^2.
	\end{equation}
	 A simple calculation shows that $|V_i|=\frac{n}{r^*}\pm \sqrt[3]{\ep_1}n$ for all $i\in[r^*]$.

	\begin{claim}\label{cl-niminside}
		For every $i\in [r^*]$ and $v\in V_i$, $d_{G^{\nim}}(v,V_i)\le\ep n$.
	\end{claim}
	\begin{proof}[Proof of Claim.]
		Assume to the contrary that there exist $i\in [r^*]$ and $v\in V_i$ such that $d_{G^{\nim}}(v,V_i)>\ep n$. 
		%Let $\xi':{[r+1]\choose \le 2}\rightarrow [k]$ be defined as follows. 
		For each $j\in [r^*]$, as $\cP$ is a max-cut, there exists a colour $\ell\in [k]$ such that $$
		d_{G^{\nim}_{\ell}}(v,V_j)\ge d_{G^{\nim}}(v,V_j)/k \ge d_{G^{\nim}}(v,V_i)/k\ge\ep n/k=:m.
		$$ 
So, for $j\in [r^*]$, let $Z_j\subseteq N_{G^{\nim}_{\ell}}(v,V_j)$ be any subset of size $m$. %$\xi'(r+1,j):=\ell$ 
		%Set $Z:=\cup_{j\in[r]} Z_j$. 
		We have 
		% Since $e(G^{\nim}[\cP])\ge t(n,r)-\ep_1n^2$, as $1/k\gg\ep\gg\ep_1\gg\ep_2$,
		\begin{eqnarray}\label{eq-check}
				e(\overline{G^{\nim}}[Z_1,\ldots,Z_{r^*}])\le e(\overline{G^{\nim}}[V_1,\dots,V_{r^*}])\stackrel{(\ref{eq-p-cr})}{\le} \ep_1n^2.
				%\le \ep_1\left(r\cdot \frac{\ep n}{k}\right)^2= \ep_1|Z|^2.
		\end{eqnarray} 
		Let $\rho:= \rho(k,r^*,N)$, where $\rho$ is the function from the Partite Ramsey Lemma (Lemma~\ref{lm:PRL}). For $i\in [r^*]$, let $Y_i$ be a random $\rho$-subset of $Z_i$, chosen uniformly and independently at random. By~\eqref{eq-check}, the expected number of missing cross-edges in $G^\nim[Y_1,\dots,Y_{r^*}]$ is at most 
		$$
		\ep_1n^2 \left({m-1\choose \rho-1}\Big/{m\choose \rho}\right)^2=\ep_1\left(\frac{\rho k}{\ep}\right)^2<1.
		$$ Thus there is a choice of  the $\rho$-sets $Y_i$'s such that $G^\nim[Y_1,\dots,Y_{r^*}]$ has no missing cross-edges. By the definition of $\rho$, there are $N$-sets $U_1\subseteq Y_1,\dots,U_{r^*}\subseteq Y_{r^*}$ and  a colouring $\xi:{[r^*]\choose \le 2}\to[k]$ such that
		$\xi(U_1,\dots,U_{r^*})\subseteq \phi$.  
		
		Note that $\xi$ is feasible. Indeed, if we have, for example, $\xi(ij)=\xi(i)=:c$, then by taking one vertex of $U_j$ and all $N$ vertices of $U_i$
		we get a colour-$c$ copy of $K_{N+1}$. However, since $N+1\ge v(H_c)$, every edge of this clique is in an $H_c$-subgraph, contradicting the fact that all pairs in the complete bipartite graph $K[U_i,U_j]$ are \NIM-edges.
		
		Consequently, as $(H_1,\ldots,H_k)$ is nice, $\xi$ must assign the same colour to all singletons, say colour~$1$.
		By construction, the vertex $v$ is monochromatic into each $Z_i\supseteq U_i$. So we can take
		$\xi':{[r^*+1]\choose \le 2}\to [k]$ such that 
		$\xi'(U_1,\dots,U_{r^*},\{v\})\subseteq \phi$, where we additionally let $\xi'(r^*+1):=1$.
		As $r^*=r^*(H_1,\ldots,H_k)$, the colouring $\xi'$ violates~(P\ref{item-r*1}) or~(P\ref{item-r*2}). This violation has to include the vertex $r^*+1$ since the restriction of $\xi'$ to ${{[r^*]\choose \le 2}}$ is the feasible colouring~$\xi$. We cannot have $i\in [r^*]$ with $\xi'(i,r^*+1)=1$ because otherwise $U_i\cup \{v\}$ is an $(N+1)$-clique coloured 1 under $\phi$, a contradiction to all pairs between $Z_i\supseteq U_i$ and $v$ being \NIM-edges. Therefore, there exists an edge-monochromatic  homomorphic copy of $H_j$ of colour $j$, say $F$, with $r^*+1\in V(F)$. By the definition of homomorphism-criticality, there exists a homomorphism $g:V(H_j)\rightarrow V(F)$ such that $|g^{-1}(r^*+1)|=1$. Therefore, we can find an edge-monochromatic copy of $H_j$ in colour $j$, with $g^{-1}(r^*+1)$ mapped to $v$, and all the other vertices of $H_j$ mapped to vertices in $U_1\cup\ldots\cup U_{r^*}$, a contradiction to all pairs between this set and $v$ being \NIM-edges.
\end{proof}
	
	We next show that all pairs inside a part get the same colour under~$\phi$.

	\begin{claim}\label{cl-allinside}
		For any $p\in [r^*]$ and any $u_1u_2,\,u_3u_4\in {V_p\choose 2}$, we have $\phi(u_1u_2)=\phi(u_3u_4)$.
	\end{claim}
	\begin{proof}[Proof of Claim.] Suppose on the contrary that $u_1,\dots,u_4\in V_p$ violate
	the claim. Without loss of generality, let $p=r^*$. Let $U:=\{u_1,\dots,u_4\}$.
		By~\eqref{eq-minnimdeg}, Claim~\ref{cl-niminside} and the fact that $|V_r|=n/r\pm \sqrt[3]{\ep_1}\,n$,  all but at most $2\ep n$ edges from any $u\in V_{r^*}$ to $V\setminus V_{r^*}$ are \NIM-edges.
		For $i\in[r^*-1]$ (resp.\ $i=r^*$), define $Z_{i}\subseteq V_{i}$ to be a largest subset of $\cap_{j=1}^4 N_{G^{\nim}}(u_j,V_i)$ (resp.\ $V_{r^*}\setminus U$) 
	 with the same colour pattern to $U$, i.e., for all $x,x'\in Z_{i}$ and $j\in [4]$ we have  $\phi(u_jx)=\phi(u_jx')$. 
	 %We will later write $\phi(u_j,Z_i):=\phi(u_1x)$ for some (any) $x\in Z_i$. 
	 By the Pigeonhole Principle, we have for $i\in [r^*-1]$ that
		$$|Z_{i}|\ge \frac{|\cap_{j=1}^4 N_{G^{\nim}}(u_j,V_i)|}{k^4}\ge %\frac{|V_{i}|-d_{\overline{G^{\nim}}}(u_1,V_i)-d_{\overline{G^{\nim}}}(v_1,V_i)}{k^2}\ge
		  \frac{|V_{i}|-4\cdot 2\ep n}{k^4}\ge \frac{n}{2r^*k^4}.$$ 
		  Also, $|Z_{r^*}|\ge (|V_{r^*}|-4)/k^4\ge n/(2r^*k^4)$.
	
	 Similarly to the calculation after~\eqref{eq-check}, there are $N$-subsets
		$U_i\subseteq Z_i$, $i\in [r^*]$, such that $\phi$ contains the blow-up $\xi(U_1,\dots,U_{r^*})$ of some $\xi: {[r^*]\choose \le 2}\rightarrow [k]$. As in the proof of Claim~\ref{cl-niminside}, $\xi$ is feasible and assigns the same colour, say $1$, to all singletons.
	    Since $\phi(u_1u_2)\not=\phi(u_3u_4)$, assume that e.g.\ $\phi(u_1u_2)\not=1$.
		 
		 We define the colouring $\xi':{[r^*+1]\choose \le 2}\rightarrow[k]$ so that $\xi'(U_1,\dots,U_{r^*-1},\{u_1\},\{u_2\})\subseteq \phi$, where additionally we let both $\xi'(r^*)$ and $\xi'(r^*+1)$ be~$1$. Note that $\xi'(r^*,r^*+1)=\phi(u_1u_2)$. Also, observe that we do not directly use the part $U_{r^*}$ when defining $\xi'$: the role of this part was to guarantee that $\xi$ is monochromatic on all singletons.		
		By the definition of $r^*$, the colouring $\xi'$ violates~(P\ref{item-r*1}) or~(P\ref{item-r*2}).
		
		Suppose first that $\xi'$ violates~(P\ref{item-r*2}), that is there is a pair $ij\in {[r^*+1]\choose 2}$ with $\xi'(ij)=1$. Since $\xi'(r^*,r^*+1)=\phi(u_1u_2)\not=1$, we have $\{i,j\}\not=\{r^*,r^*+1\}$. Also, we cannot have
		$i,j\in [r^*-1]$, because $\xi'$ coincides  on ${{[r^*-1]\choose \le 2}}$ with the feasible colouring $\xi$. So we can assume by symmetry that $i\in [r^*-1]$ and $j=r^*$. However, then the vertex $u_1$ is connected by \NIM-$1$-edges to the colour-$1$ clique on the $N$-set $U_i$, a contradiction. 
		
		We may now assume that the colouring $\xi'$ violates~(P\ref{item-r*1}). Let this be witnessed
		by an edge-monochromatic homomorphic copy of $H_j$ of colour $j$, say $F$. If $F$ contains exactly one vertex from $\{r^*,r^*+1\}$, then by an argument similar to the last part of the proof of Claim~\ref{cl-niminside} we get a contradiction. Otherwise, if $\{r^*,r^*+1\}\subseteq V(F)$, then, by the definition of homomorphism-critical, there exists a homomorphism $g:V(H_j)\rightarrow V(F)$ such that $|g^{-1}(r^*)|=|g^{-1}(r^*+1)|=1$. Therefore, we can find an edge-monochromatic copy of $H_j$ in colour $j$, with $g^{-1}(r^*)$ (resp. $g^{-1}(r^*+1)$)  mapped to $u_1$ (resp. $u_2$), and all the other vertices of $H_j$ mapped to vertices in $U_1\cup\ldots\cup U_{r^*-1}$, a contradiction to all pairs between this set and $\{u_1, u_2\}$ being \NIM-edges.
	\end{proof}
	
	Let $i\in [r^*]$. By Claim~\ref{cl-allinside} we know that $G[V_i]$ is a monochromatic clique. Since $|V_i|\ge  \max_{j\in[k]}v(H_j)$, no pair inside $V_i$ is a \NIM-edge. Thus $G^{\nim}$ is $r^*$-partite. Our assumption $e(G^{\nim})\ge t(n,r^*)$ implies that $G^\nim$ is isomorphic to $T(n,r^*)$, as desired.
\end{proof}

We are now ready to prove the desired exact result.

\begin{proof}[Proof of Theorem~\ref{thm-exact}]  We know by Theorem~\ref{thm-ram-nim} that $\nim(n;H_1,\dots,H_k)\ge t(n,r^*)$ for all~$n$. 

On the other hand, let $n_0$ be the constant returned by Lemma~\ref{lem-thm-exact}. Let $n\ge n_0^2$ and let $\psi$ be an extremal colouring of $G:=K_n$. In order to finish the proof of the theorem it is enough to show that necessarily $G^{\nim}\cong T(n,r^*)$.

Initially, let $i=n$, $G_n:=G$ and $\phi_n:=\psi$. Iteratively repeat the following step as long as possible: if the \NIM-graph of $\phi_i$ has a vertex $x_i$ of degree smaller than $\de(T(i,r^*))$, let $\phi_{i-1}$ be the
restriction of $\phi_i$ to the edge-set of $G_{i-1}:=G_i-x_i$ and decrease $i$ by $1$. Suppose that this procedure ends with $G_m$ and $\phi_m$.

Note that, for every $i\in\{m+1,\dots,n\}$, we have that 
 \begin{eqnarray*}
 t(i-1,r^*)&=&t(i,r^*)-\de(T(i,r^*)),\\
 \nim(\phi_{i-1})&\ge& \nim(\phi_{i})-\de(T(i,r^*))+1,
 \end{eqnarray*}
  the latter inequality following from the fact
that every \NIM-edge of $\phi_{i}$ not incident to $x_i$ is necessarily a \NIM-edge of $\phi_{i-1}$. 
These two relations imply by induction that 
 \begin{equation}\label{eq:ImprInd}
 \nim(\phi_i)\ge t(i,r^*)+n-i,\quad\mbox{for $i=n,n-1,\dots,m$.}
 \end{equation}
 In particular, it follows that $m>n_0$ for otherwise 
 $\NIM(\phi_{n_0})$ is a graph of order $n_0$ with at least $n-n_0>{n_0\choose 2}$ edges, which is impossible. Thus Lemma~\ref{lem-thm-exact} applies to $\phi_m$ and gives that $\NIM(\phi_m)\cong T(m,r^*)$. By~\eqref{eq:ImprInd} we conclude that $m=n$, finishing the proof of  Theorem~\ref{thm-exact}.\end{proof}
%%%%%%%%%%%%%%%%%%%%%%%%%%%%%%%%%%%%%%%%%%%%%%%%%%%%%%%%
%%%%%%%%%%%%%%%%%%%%%%%%%%%%%%%%%%%%%%%%%%%%%%%%%%%%%%%%

\section{Proofs of Theorems~\ref{thm-conjfor3} and~\ref{thm-4triangles}}\label{sec-2specialcases}

Next we will show that  Conjecture~\ref{conj-r*=r} holds for the $3$-colour case.

\begin{proof}[Proof of Theorem~\ref{thm-conjfor3}] 
Take an arbitrary feasible $3$-colouring $\xi$ of ${[r]\choose \le 2}$, where $r=R(a_2,a_3)-1$.
It suffices to show that $\xi$ assigns the same colour to all the singletons in $[r]$. Indeed, suppose that $(K_{a_1},K_{a_2},K_{a_3})$ is not nice. Then there exists a feasible $3$-colouring $\xi^*$ of ${[r^*]\choose \le 2}$ that is not monochromatic on the singletons in $[r^*]$, where $r^*:=r^*(K_{a_1},K_{a_2},K_{a_3})\ge r$. Up to relabeling, we may assume that $[r]$ contains two singletons of different colours in $\xi^*$. We then arrive to a contradiction, as the restriction of $\xi^*$ on $[r]$ is also feasible. Fix now an arbitrary feasible $3$-colouring of ${[r^*]\choose \le 2}$, which assigns the same colour, say colour $i$, to all the singletons in $[r^*]$. Then due to~(P2), colour $i$ cannot appear on ${[r^*]\choose 2}$, and so $r^*\le R(a_j,a_k)-1\le r$, where $\{j,k\}=[3]\setminus\{i\}$.

%	Let $C:{r\choose \le 2}\rightarrow [3]$ be a feasible colouring. 
For $i\in [3]$, let $V_i$ be the set of vertices with colour $i$. Thus we have a partition $[r]=V_1\cup V_2\cup V_3$. For $i,j\in [3]$, let $\omega_j(V_i)$ be the size of the largest edge-monochromatic clique of colour $j$ in $V_i$. 

Observe the following properties that hold for every triple $i,j,\ell\in [3]$ of distinct indices, i.e.,\ for $\{i,j,\ell\}=[3]$. By~(P\ref{item-r*2}), the colour of every edge inside $V_i$ is either $j$ or $\ell$ while all the edges going between $V_j$ and $V_\ell$ have colour $i$. By the latter property and~(P\ref{item-r*1}), we have
		\begin{eqnarray}\label{eq-ai}
		\omega_i(V_{\ell})+\omega_i(V_j)\le a_i-1\quad\mbox{and}\quad V_j\not=\emptyset\ \Rightarrow\ \omega_i(V_{\ell})\le a_i-2.%\nonumber
		\end{eqnarray}
		
For notational convenience, define $r(n_1,\dots,n_k):=R(n_1,\dots,n_k)-1$ to be one less than the Ramsey number (i.e.\ it is the maximum order of a clique admitting a $(K_{n_1},\dots,K_{n_k})$-free edge-colouring). By the definition of $\omega_j(V_i)$, we also have
		\begin{eqnarray}\label{eq-Vi}
		|V_i|\le r(\omega_j(V_i)+1,\omega_{\ell}(V_i)+1).
		\end{eqnarray}
 Also, we will use the following trivial inequalities involving Ramsey numbers that hold for arbitrary integers $a,b,c\ge 2$: $r(a,b)+r(a,c)\le r(a,b+c-1)$ and $r(a,b)< r(a+1,b)$.

First, let us derive the contradiction from assuming that each colour $i\in [3]$ appears on at least one singleton, that is, each $V_i$ is non-empty. In order to reduce the number of cases, we allow to swap colours $1$ and $2$ to ensure that $\omega_1(V_2)\ge \omega_2(V_1)$. Thus we do not stipulate now which of $a_1$ and $a_2$ is larger. Observe that
		\begin{eqnarray}
		|V_1|+|V_2|&\overset{\eqref{eq-Vi}}{\le}& r(\omega_2(V_1)+1,\omega_3(V_1)+1)+r(\omega_1(V_2)+1,\omega_3(V_2)+1)\nonumber\\
		&\le& r(\omega_1(V_2)+1,\omega_3(V_1)+1)+r(\omega_1(V_2)+1,\omega_3(V_2)+1)\nonumber\\
		&\le&r(\omega_1(V_2)+1,\omega_3(V_1)+\omega_3(V_2)+1)\ \stackrel{\eqref{eq-ai}}{\le}\ r(\omega_1(V_2)+1,a_3).\label{eq-V12}
		\end{eqnarray}
 Hence, we get
		\begin{eqnarray}
		r\ =\ |V_1|+|V_2|+|V_3|&\overset{\eqref{eq-Vi},\eqref{eq-V12}}{\le}& r(\omega_1(V_2)+1,a_3)+r(\omega_1(V_3)+1,\omega_2(V_3)+1)\nonumber\\
		&\stackrel{\eqref{eq-ai}}{\le}& r(\omega_1(V_2)+1,a_3)+r(\omega_1(V_3)+1,a_2-1)\nonumber\\
		&<& r(\omega_1(V_2)+1,a_3)+r(\omega_1(V_3)+1,a_3)\nonumber\\
		&\le&r(\omega_1(V_2)+\omega_1(V_3)+1,a_3)\ \stackrel{\eqref{eq-ai}}{\le}\ r(a_1,a_3)\ \le\ r.\nonumber
		\end{eqnarray}

The above contradiction shows that, for some $\ell\in [3]$, the part $V_\ell$ is empty. Let $\{i,j,\ell\}=[3]$; thus $[r]=V_i\cup V_j$. It remains to derive a contradiction by assuming that each of $V_i$ and $V_j$ is non-empty.
By the symmetry between $i$ and $j$, we can assume that $\omega_j(V_i)\ge \omega_i(V_j)$. Then we have
	\begin{eqnarray*}
		r=|V_i|+|V_j|&\stackrel{\eqref{eq-Vi}}{\le}& r(\omega_j(V_i)+1,\omega_{\ell}(V_i)+1)+r(\omega_i(V_j)+1,\omega_{\ell}(V_j)+1)\\
		&\le&r(\omega_j(V_i)+1,\omega_{\ell}(V_i)+1)+r(\omega_j(V_i)+1,\omega_{\ell}(V_j)+1)\\
		&\le &r(\omega_j(V_i)+1,\omega_{\ell}(V_i)+\omega_{\ell}(V_j)+1)\\
		&\stackrel{\eqref{eq-ai}}{\le}& r(a_j-1,a_{\ell})\ <\ r(a_j,a_{\ell})\ \le\ r,
		\end{eqnarray*}	
		which is the desired contradiction that finishes the proof of Theorem~\ref{thm-conjfor3}.
\end{proof}

Next, let us present the proof that $r^*(3,3,3,3)=16$, the only non-trivial 4-colour case that we can solve.

\begin{proof}[Proof of Theorem~\ref{thm-4triangles}]
	Let $\xi:{[16]\choose \le 2}\rightarrow [4]$ be an arbitrary feasible colouring. It is enough to show that all singletons in $[16]$ get the same colour. For every $i\in[4]$, let $V_i$ denote the set of vertices of colour $i$. Suppose there are at least two different colours on the vertices, say $V_3,V_4\neq\emptyset$. As $5$ does not divide $16$, there exists at least one class, say $V_3$, of size not divisible by $5$, i.e.,~$|V_3|\not\equiv 0\pmod 5$. Choose an arbitrary vertex $v\in V_4$. Since $\xi$ is a feasible colouring, by~(P\ref{item-r*2}) the edges incident to $v$ cannot have colour $\xi(v)=4$. We can then partition $[16]\setminus\{v\}=\cup_{j\in[3]} W_j$, where $W_j:=\{u:~ \xi(uv)=j\}$. Let $j\in[3]$. By~(P\ref{item-r*1}) and~(P\ref{item-r*2}), colour $j$ is forbidden in ${W_j\choose \le 2}$. Then by Theorem~\ref{thm-conjfor3}, $|W_j|\le r^*(K_3,K_3,K_3)=R(3,3)-1=5$. Since $\sum_{j\in[3]}|W_j|=15$, we have that $|W_j|=5$ for every $j\in[3]$. Again by Theorem~\ref{thm-conjfor3}, all vertices in $W_j$ should have the same colour. Recall that $v\in V_4$, so $V_3\subseteq \cup_{j\in[3]}W_j$ and consequently $V_3$ is the union of some $W_j$'s.
	This contradicts $|V_3|\not\equiv 0\pmod 5$.
\end{proof}

%%%%%%%%%%%%%%%%%%%%%%%%%%%%%%%%%%%%%%%%%%%%%%%%%%%%%%
%%%%%%%%%%%%%%%%%%%%%%%%%%%%%%%%%%%%%%%%%%%%%%%%%%%%%%

%%%%%%%%%%%%%%%%%%%%%%%%%%%%%%%%%%%%%%%%%%%%%%%%%%%%%%%

\section{Concluding remarks}\label{sec-concluding}
\begin{itemize}
\item As pointed out by a referee, the function $\nim(n;H_1,\ldots,H_k)$ is related to that of $\ex^r(n;H_1,\ldots,H_r)$, which is the maximum size of an $n$-vertex graph $G$ that can be $r$-edge-coloured so that the $i$-th colour is $H_i$-free for all $i\in [r]$. Indeed, we have the following lower bound:
$$ \nim(n;H_1,\dots,H_k)\ge \max_{i\in [k]}\,\ex^{k-1}(H_1,\ldots,H_{i-1},H_{i+1},\ldots,H_k).$$
It is not inconceivable that the equality holds above if $n\ge n_0(H_1,\ldots,H_k)$. Theorems~\ref{thm-exact} and~\ref{thm-weaklyred} give classes of instances, when we have equality above. We refer the readers to Section 5.3 of~\cite{HancockStadenTreglown17arxiv} for more on the function $\ex^r$.	
	
\item The Ramsey variant $r^*$ introduced here is related to the version of Ramsey numbers studied by~Gy\'arf\'as, Lehel, Schelp and Tuza~\cite{GyarfasLST87}. In particular, Proposition 5 in~\cite{GyarfasLST87} states that $r^*(K_3,K_3,K_3,K_3)=16$, which is the consequence of the fact that $(K_3,K_3,K_3,K_3)$ is nice from Theorem~\ref{thm-4triangles}.
	
\item We prove in Theorem~\ref{thm-tree} that for any tree $T$, $\nim_{k}(n;T)=(k-1)\ex(n,T)+O_T(1)$. Let $T$ be an $(h+1)$-vertex tree and suppose that the Erd\Ho s-S\'os conjecture holds, i.e.~$\ex(n,T)\le (h-1)n/2$. Then for each $n\ge k^2h^2$ with $h| n$, we can get rid of the additive error term in the lower bound, namely, it holds that $\nim_{k}(n;T)\ge (k-1)\ex(n,T)$. This directly follows from
known results on graph packings. We present here a short self-contained proof (with a worse bound on $n$). Let $F$ be the disjoint union of $n/h$ copies of $K_h$. Let $f_i: V(F)\rightarrow [n]$, $i\in[k-1]$, be $k-1$ arbitrary injective maps and let $F_i$ be the graph obtained by mapping $F$ on $[n]$ via $f_i$. It suffices to show that we can modify $f_i$'s to have $E(F_i)\cap E(F_j)=\emptyset$ for any $ij\in{[k-1]\choose 2}$. Indeed, then the lower bound is witnessed by colouring $e\in E(K_n)$ with colour-$i$ if $e\in E(F_i)$, for each $i\in[k-1]$, and with colour-$k$ otherwise. Suppose that there is a ``conflict'' $uv\in E(F_i)\cap E(F_j)$. Let $F^*:=\cup_{i\in[k-1]}F_i$. Note that $\Delta(F^*)\le (k-1)(h-1)$. As $n>\Delta(F^*)^2+1$, there exists a vertex $w$ that is at distance at least 3 from $v$. We claim that switching $v$ and $w$ in $f_i$ will remove all conflicts at $v$ and $w$. If true, one can then repeat this process till all conflits are removed to get the desired $f_i$'s. Indeed, suppose that after switching $v$ and $w$, there is a conflict $wz\in E(F_i)\cap E(F_{\ell})$ for some $z\in N_{F_i}(v)$ and $\ell\in[k-1]\setminus\{i\}$. Then $w,z,v$ form a path of length $2$ in $F^*$, contradicting the choice of $w$.

It would be interesting to prove a matching upper bound, i.e.~to show that 
$$\nim_{k}(n;T)=(k-1)\ex(n,T)$$ 
for every tree $T$ and sufficiently large $n$. Note that equality above need not be true when $T$ is a forest. Indeed, consider $M_2$, the disjoint union of two edges. Recall that $\ex(n,M_2)=n-1$. For any $k\ge 3$ and $n\ge 4k$, we have that $\nim_{k}(n;M_2)=(k-1)\ex(n,M_2)-\frac{1}{2}(k-1)(k-2)$. Indeed, for any $k$-edge-colouring $\phi$ of $K_n$, one colour class, say colour-1, has size at least ${n\choose 2}/k\ge 2(n-1)$. As every edge share endpoints with at most $2n-4$ other edges, we see that every colour-1 edge is in a copy of $M_2$. Thus, $\nim(\phi)\le \ex^{k-1}(n,M_2,\ldots,M_2)=\sum_{i=0}^{k-2}(n-1-i)=(k-1)\ex(n,M_2)-\frac{1}{2}(k-1)(k-2)$, as desired.
\end{itemize}

%%%%%%%%%%%%%%%%%%%%%%%%%%%%%%%%%%%%%%%%%%%%%%%%%%%%%%%

\section*{Acknowledgement}

We would like to thank the referees for pointing out a flaw in the original draft, and for their careful reading, which has greatly improved the presentation of this paper.

\bibliography{NIM-2nd-revision-Jan26}
\end{document}